\documentclass[11pt,twoside]{amsart}

      \usepackage{amssymb}
      \usepackage{graphicx}
      \usepackage{dcolumn,indentfirst,cite,color,hyperref}
     \usepackage{pgfplots}
      \usepackage{pgf}
    \usepackage{tikz}
    \usetikzlibrary{arrows, decorations.pathmorphing, backgrounds, positioning, fit, petri, automata}
 
\definecolor{yellow1}{rgb}{1,0.8,0.2} 

\theoremstyle{plain}
 \newtheorem{thm}{Theorem}[section]
 
\newtheorem{lem}[thm]{Lemma}
\newtheorem{que}[thm]{Question}
\newtheorem{cor}[thm]{Corollary}
\newtheorem{pro}[thm]{Proposition}
\newtheorem{rmk}[thm]{Remark}

\newtheorem{defi}[thm]{Definition}
\newcommand{\bess}{\begin{eqnarray*}}
\newcommand{\eess}{\end{eqnarray*}}
\input xy
\xyoption{all}

\usepackage{fancyhdr} 
\begin{document}

\author{\textsc{Xiaoguang Wang}}
\address{School of Mathematical Sciences, Zhejiang University, Hangzhou, 310027, China}
\email{wxg688@163.com}

\author{\textsc{Yongcheng Yin}}
 \address{School of Mathematical Sciences, Zhejiang University, Hangzhou, 310027, China}
 \email{yin@zju.edu.cn}






   \title[]{Global topology of hyperbolic components I:   Cantor circle case}


   \begin{abstract} The   hyperbolic components  in the moduli space ${M}_d$  of  degree $d\geq2$ rational maps are 
    mysterious and  fundamental topological objects. For those in the connectedness locus, they are known to be the finite quotients of 
    the Euclidean space $\mathbb{R}^{4d-4}$. In this paper,  we 
  study the  hyperbolic components in  the disconnectedness locus and with minimal complexity: those in the Cantor circle locus.
   We show that each of them  is a  finite quotient of the space $\mathbb{R}^{4d-4-n}\times\mathbb{T}^{n}$,
  where $n$ is determined by the dynamics. The proof relates Riemann surface theory (Abel's Theorem), dynamical system  and algebraic topology.
   \end{abstract}

   \subjclass[2010]{Primary 37F45; Secondary 37F10, 37F15}

   \keywords{global topology, hyperbolic component, moduli space}



   \date{\today}


   \maketitle

\section{Introduction and main theorem} \label{intro}

This is the first of a series of papers which will be devoted to a study of the global topology of the hyperbolic components in the {\it disconnectedness locus}, in the moduli space  of rational maps. The 
 problem  has various challenging cases and lies in the crossroad of many subjects: Riemann surface theory, dynamical systems, algebraic topology, etc.  It is beyond the authors' ability to treat all the cases in one single paper, so a series of papers will fit the project. In the current paper, we will deal with the hyperbolic components with `minimal' complexity: those in the Cantor circle locus.
  We will illustrate how different subjects 
 interact in this  situation. Our treatment in this case sheds lights on the strategy to deal with the general case.
 
 To set the stage,
let's begin with some basic definitions and  motivations.

\

Let ${\rm Rat}_d$ be the space of  rational maps $f: \widehat{\mathbb{C}}\rightarrow  \widehat{\mathbb{C}}$ of degree $d\geq 2$.  This space is  naturally  parameterized as  
$${\rm Rat}_d\simeq \mathbb{P}^{2d+1}(\mathbb{C})\setminus V({\rm Res}),$$
where $V ({\rm Res})$ is the hypersurface of  the irreducible   pairs $(P, Q)$ defining $f=P/Q$, for which the resultant vanishes.  The moduli space 
$$M_d={\rm Rat}_d/ {\rm PSL}(2, \mathbb{C})$$ is   ${\rm Rat}_d$ modulo the action by conjugation of the group ${\rm PSL}(2, \mathbb{C})$ of M\"obius transformations. There is a natural projection $\pi: {\rm Rat}_d\rightarrow {M}_d$ sending a rational map $f$ to its M\"obius conjugate class $\langle f \rangle$. The topology  on $M_d$ is the quotient topology induced by $\pi$. A point $\langle f \rangle \in M_d$ is also called a map if no confusion arises.


 A  rational  map $f$ is  {\it hyperbolic}  if all critical points are attracted, under iterations,   to the attracting cycles of $f$.   
  We say  that $\langle f \rangle$ is {\it hyperbolic} if  $f$ is hyperbolic. It's known that in the moduli space $M_d$, the set $M_d^{hyp}$  of all  hyperbolic maps 
 is open, and  conjecturally  dense.  
 A connected component of $M_d^{hyp}$ is called a {\it hyperbolic component}.  It is shown by DeMarco \cite{D} that  every hyperbolic  component (in any holomorphic family parameterized by a complex manifold, implying that in $M_d$) is a domain of holomorphy.
 Even though, the shapes of general hyperbolic components  are  still mysterious.
 Therefore,  a fundamental problem naturally arises:
 
 \begin{que}\label{shape} What is the global topology  of the hyperbolic component?
\end{que} 

 
 An extremal  case is that the hyperbolic components are in the  {\it connectedness locus}, 
the collection of all  maps whose Julia sets are connected. Working within the  polynomial moduli space of degree $d$,  Milnor \cite{M} shows that every  hyperbolic component  in the connectedness locus is diffeomorphic to the topological  cell $\mathbb{R}^{2d-2}$.
Milnor's result has a slightly different statement when applying to the rational moduli space $M_d$.
Due to the possible symmetries of rational maps, the  hyperbolic components in the connectedness locus in $M_d$ are 
 actually the finite quotients of $\mathbb{R}^{4d-4}$, compare \cite[Section 9]{M}. In particular, 
 in the quadratic case, the topology of the hyperbolic components  was described by Rees \cite{R}.

  For the hyperbolic components  in the {\it disconnectedness locus}, consisting of the maps for which the  Julia sets are  disconnected,    Makienko \cite{Ma} showed that each of them is unbounded in the moduli space $M_d$.  Further studies of these components were previously focused on the polynomial {\it shift locus}, see for example 
 \cite{BDK, DM, DP1, DP2, DP3}.
 In the rational case, however,  very little is known about their  global topology.

Our main purpose   is to understand the global topology of these hyperbolic components.
In this paper,  we consider the hyperbolic components in the  {\it Cantor circle locus} (subset of the disconnectedness locus), consisting of maps $\langle f\rangle$ for which the Julia set $J(f)$ is  homeomorphic to the Cartesian product of  a Cantor set 
 $\mathbf{C}$  and  the circle ${S}^1=\mathbb{R}/\mathbb{Z}$:
 $$J(f)\simeq \mathbf{C}\times {S}^1 ,$$
 here a {\it Cantor set} $\mathbf{C}$ is a  totally disconnected  perfect set  in $\mathbb{R}$. 
These hyperbolic components have the `minimal' complexity among all those in the  disconnectedness locus, in the sense that in the dynamical space, each Fatou component of a representative map is either simply connected or doubly connected.
This  non-trivial case naturally serves as the starting point of our exploration.
In fact, our study in this case  involves ideas and techniques from Riemann surface theory (e.g. Abel's Theorem), dynamical systems (e.g. deformations and combinations of rational maps) and algebraic topology.  This enlightens the way to understand the general case, as will be illustrated in the forthcoming papers.

Our main result  characterizes  the global topology of these components:


\begin{thm}\label{cantor-circle-locus}  Every  hyperbolic component  $ \mathcal{H}\subset M_d$ in the Cantor circle locus  is a finite quotient of $\mathbb{R}^{4d-4-n}\times\mathbb{T}^{n}$, where $n$ is the number of the critical annular Fatou components of a representative map in $\mathcal{H}$. 
%
 \end{thm}

We remark that  the Cantor circle locus is empty when $d\leq 4$, therefore Theorem \ref{cantor-circle-locus} implicitly 
requires that $d\geq 5$. 
It is also necessary to point out that even in the case that the quotient map  is  a covering map, 
the  hyperbolic component $\mathcal{H}$  is not necessarily homeomorphic to $\mathbb{R}^{4d-4-n}\times\mathbb{T}^{n}$.
In fact, 
classifying the spaces finitely covered by  $\mathbb{T}^n$ (or the more delicate case $\mathbb{R}^{4d-4-n}\times \mathbb{T}^n$) is an important subject in topology,  and is related to the theory of  crystallographic  groups and
   Hilbert's eighteenth problem \cite{M2}.   This is beyond the scope of the paper. The readers may refer to  \cite{B1,B2,Ch,F,FG,FH,Hi} and Section \ref{appendix}.

%
%
%







The main point in the proof of Theorem \ref{cantor-circle-locus} is to show that a marked version $\mathcal{F}=\mathcal{F}(\mathcal{H})$ of $\mathcal{H}$ is homeomorphic to $\mathbb{R}^{4d-4-n}\times\mathbb{T}^{n}$ (as is shown in Section \ref{top_marked}).
This will be built on two key steps.

Key \textbf{Step 1}    is to characterize the  proper holomorphic maps from the annulus to the disk, and the space of all such maps.  To this end, fix a  number $r\in(0,1)$,  let $\mathbb{D}=\{z\in \mathbb{C}; |z|<1\}$,  $\mathbb{A}_r=\{z\in \mathbb{C};  r<|z|<1\}$ and $C_r=\{z\in \mathbb{C}; |z|=r\}$ be the inner boundary of $\mathbb{A}_r$.  Let $p_1, p_2, \cdots,  p_e$ be $e\geq 2$ points in 
$\mathbb{A}_r$, not necessarily distinct, and let $\delta\in[1,e)$ be an integer.
 Our main result in this step is the following:
  
 \begin{thm} \label{properAD}
 There is a  proper holomorphic map   $f: \mathbb{A}_r \rightarrow \mathbb{D}$ of degree 
$e$ with $f^{-1}(0)=\{p_1, \cdots, p_e\}$ and ${\deg}(f|_{C_r})=\delta$  if and only if 
$$|p_1p_2\cdots p_e|=r^{\delta}.$$ 
When exists, the proper map $f$ is unique if we further require that $f(1)=1$. In this case, $f$ can be written uniquely as
$$f(z)= {z^{-\delta}} B_{0}(z)\prod_{j\geq 1} \Big(B_{j}(z)B_{-j}(z)\Big),  \  z\in  \mathbb{A}_r ,$$
where 
$$B_{j}(z)= \Bigg(\prod_{k=1}^e\frac{1-\overline{p_k}r^{2j}}{1-{p_k} r^{2j}}\Bigg) \Bigg(\prod_{k=1}^e  \frac{z-p_k r^{2j}}{1-\overline{p_k} r^{2j}z} \Bigg), \ j\in\mathbb{Z}.$$
Moreover, the space of all proper holomorphic maps  $f: \mathbb{A}_r \rightarrow \mathbb{D}$ of degree $e$ with  ${\deg}(f|_{C_r})=\delta$, 
$f(1)=1$, and with $r$ ranging over all numbers in $(0,1)$,  is homeomorphic to $\mathbb{R}^{2e-1}\times S^1$.
\end{thm}

Theorem \ref{properAD} generalizes the well-known Blaschke products, therefore has an independent interest. 
Its statement will be cut into several pieces, whose proofs are given in Sections \ref{annu-disk} and \ref{model}, respectively.  Abel's Theorem for principal divisors plays an important role in the proof of existence part of Theorem \ref{properAD} and its generalization to multi-connected domains in  Section \ref{muti-disk}.

\vspace{5pt}

Key \textbf{Step 2}  is to study the `abstract'   hyperbolic component $\mathcal{H}$ via some concrete spaces. To explain the strategy, let $n$ be the number of critical annular Fatou components of a representative map in $\mathcal{H}$. We associate $\mathcal{H}$ with a mapping scheme $\sigma$, a partition vector $\boldsymbol{d}=(d_1,\cdots, d_{n+1})\in\mathbb{N}^{n+1}$ of $d$, see Section \ref{dynamics} for precise definitions.
Then  $\sigma$ and $\boldsymbol{d}$ naturally 
 induce
 \begin{itemize}
 \item
  a  family of marked   rational maps $\mathcal{F}^{\chi_0}_{\sigma, \boldsymbol{d}}\subset{\rm Rat}_d$, where ${\chi_0}$ is related to the marking information;
  \item
 a  space of model maps $\mathbf{M}_\sigma$, and 
 \item
 two projections 
 $\rho:  \mathcal{F}^{\chi_0}_{\sigma,\boldsymbol{d}} \rightarrow \mathbf{M}_\sigma$ 
 and $\boldsymbol{p}: \mathcal{F}^{\chi_0}_{\sigma,\boldsymbol{d}} \rightarrow M_d$.
 \end{itemize}
 
 The model space $\mathbf{M}_\sigma$ is known to be homeomorphic to $\mathbb{R}^{4d-4-n}\times\mathbb{T}^{n}$ (see Corollary \ref{model-space1}). In order to understand the topology of (each component of) $\mathcal{F}^{\chi_0}_{\sigma, \boldsymbol{d}}$, we need study the property  of $\rho$. 
   Our main result in this step is

 \begin{thm} \label{covering0}   The map $\rho:  \mathcal{F}^{\chi_0}_{\sigma,\boldsymbol{d}} \rightarrow \mathbf{M}_\sigma$ is a   covering map of degree  $${\rm deg}(\rho)=\bigg(1-\sum_{k=1}^{n+1}\frac{1}{d_k}\bigg){\rm lcm}(d_1, \cdots, d_{n+1}) .$$
 In particular,   $\rho$ is a homeomorphism if and only if
 $$\sum_{k=1}^{n+1}\frac{1}{d_k}+\frac{1}{{\rm lcm}(d_1, \cdots, d_{n+1})}=1.$$ 
\end{thm}

Theorem \ref{covering0} completely answers the question: Given a generic holomorphic model map, how many rational maps (up to M\"obius conjugation) realize it?  The answer is exactly ${\rm deg}(\rho)$.

The covering property of $\rho$ is proven in Section \ref{cov}, using quasi-conformal surgery.
The mapping degree of $\rho$ is given in Section  \ref{finite}, using an idea of twist deformation, due to   Cui \cite{C}. 


By Theorem \ref{covering0} and  using an  algebraic topology argument, we prove in Section  \ref{top_marked} that each marked hyperbolic component of $\mathcal{F}^{\chi_0}_{\sigma,\boldsymbol{d}}$ is homeomorphic to $\mathbb{R}^{4d-4-n}\times\mathbb{T}^{n}$.  Then by the finiteness property of $\boldsymbol{p}$, we will complete the proof of Theorem 
\ref{cantor-circle-locus} in Section \ref{global-hyp}.
      
Section \ref{appendix} is an appendix of some supplementary materials.  

%
%
%
%
%
%
%
%
%
%
\vspace{5pt}
 
\noindent\textbf{Acknowledgement.} This paper  grows up based on discussions with many people.
The first author would like to thank especially {\it Guizhen Cui},  for  his idea of  twist deformation \cite{C}, which  leads  to the proof  of Theorem \ref{covering}; {\it Allen Hatcher}, for providing us examples, not homeomorphic to,  but finitely  covered by  $\mathbb{T}^n$ and with Abelian transformation group;  {\it John Milnor}, for  his  paper \cite{M} as a persistent source of  inspiration and for  recommendation of the references \cite{B1, B2}; {\it   Kevin Pilgrim}, for helpful discussions on the correct formulation of the space $\mathcal{F}_{\sigma, \boldsymbol{d}}^{\chi_0}$ of marked rational maps.

%

  \section{Dynamics} \label{dynamics}
 
 
 In this section, we present some basic dynamical properties of the hyperbolic rational maps whose
  Julia sets are Cantor set of circles.
  Examples of such rational maps were  given by many people, e.g. \cite{Mc,Sh,DLU,QYY}.
 
  \begin{figure}[h] 
\begin{center}
\includegraphics[height=6cm]{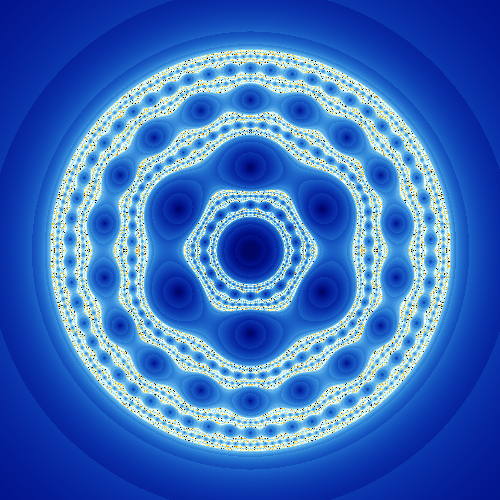}
 \caption{Example: the Julia set of $f(z)=z^3+cz^{-3}$,  when $c$ is small, is a Cantor set of circles.}
\end{center}
\end{figure}
 Let $f\in {\rm Rat}_d$ be such a map. Note that each Fatou component  of $f$ is either simply connected or doubly connected, and that the number of annular Fatou components containing critical points is  finite. 
One may also observe that there are exactly  two simply connected Fatou 
 components of $f$.  By  suitable normalization, 
 we assume one contains $0$ and  the other contains $\infty$.
   They  are  denoted by   $D_0$ and $D_\infty$,  respectively. 
 
 \begin{figure}[h] 
\begin{center}
\includegraphics[height=5cm]{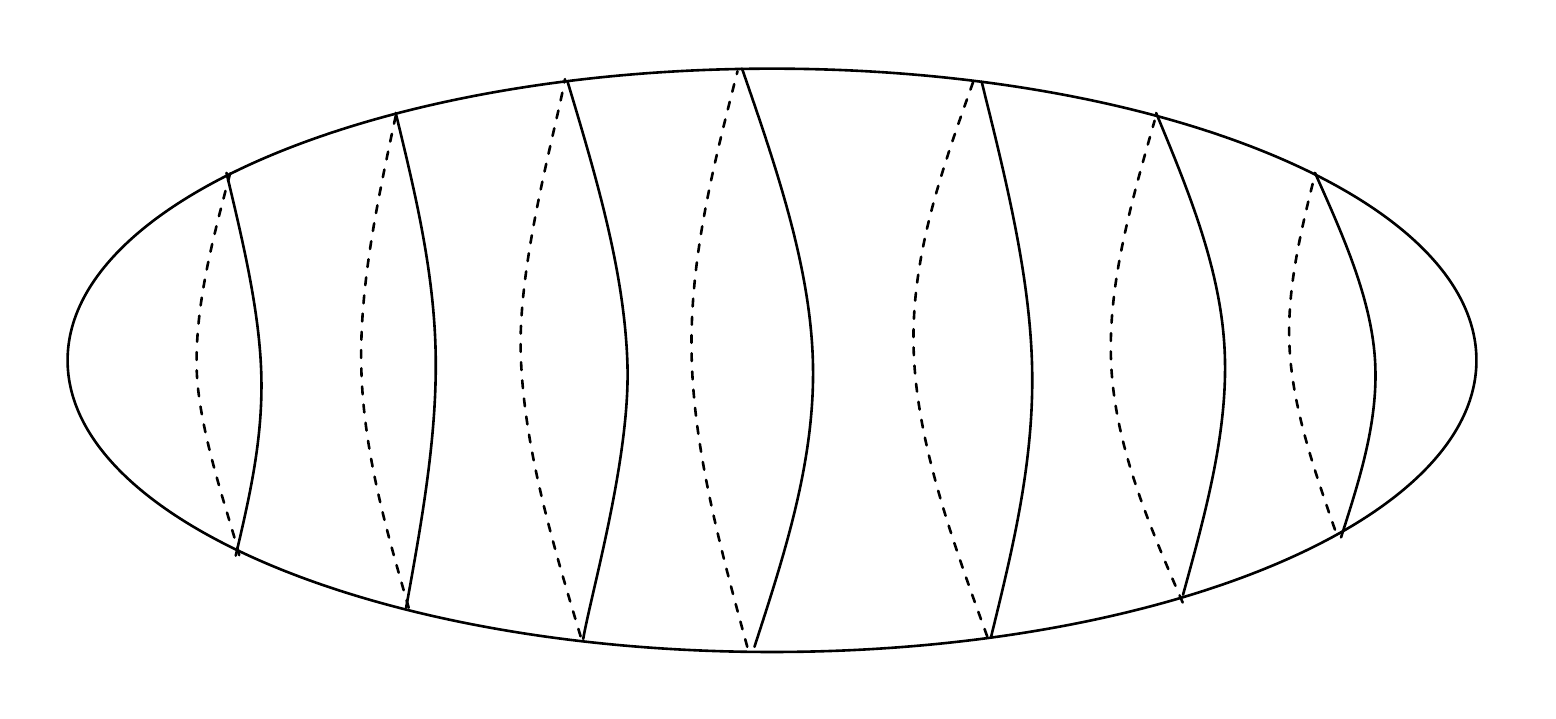}
 \put(-290, 70){$\bullet \   0$} \put(-270,70){$D_0$}  \put(-247, 70){$A_1$} \put(-205,70){$D_1$} \put(-170,70){$A_2$} \put(-140,70){$\cdots$}
  \put(-88,70){$D_n$}  \put(-65,70){$A_{n+1}$}  \put(-35,57){$D_\infty$}  \put(-25,70){$\bullet \ \infty$}
 \caption{Domains in the dynamical plane of a Cantor circle  map.   $D_1,\cdots, D_n$  are (critical)  annular Fatou components but  $A_k$'s are not.}
\end{center}
\end{figure}
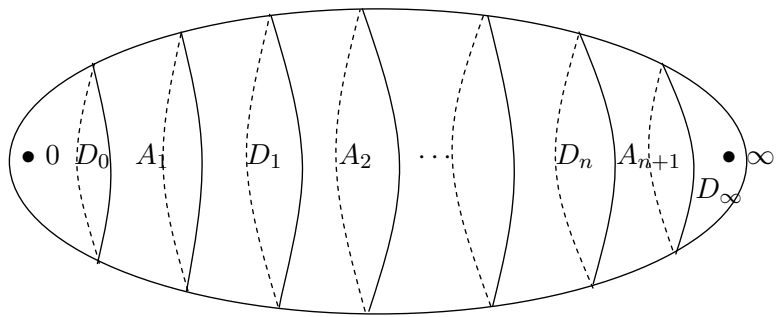
The annulus  $A=\widehat{\mathbb{C}}-\overline{D}_0\cup \overline{D}_\infty$  contains no critical value of $f$,  therefore each connected component of  $f^{-1}(A)$
is again an annulus. These components  are denoted by $A_1, A_2, \cdots, A_{n+1}$, numbered so that $A_{k+1}$ and  $\infty$ are in the same component of $\mathbb{\widehat{C}}\setminus A_k$.   See Figure 2 for the arrangement of these sets.  
Let $D_k$ be the domain  lying in between $A_k$ and $A_{k+1}$, $1\leq k \leq n$. Clearly, $D_k$ is a critical annular Fatou component.
 The collection  $\boldsymbol{F}(f)$ of  all  Fatou components  containing critical or post-critical points
 is  
 $$\boldsymbol{F}(f)=\{D_0, D_\infty, D_1, \cdots, D_n\}.$$

\subsection{Mapping scheme}

The map $f$ induces a  self map $f_*$  of $\boldsymbol{F}(f)$ defined by $f_*(D_j)=f(D_j)$. The pair $(\boldsymbol{F}(f), f_*)$  is called a {\it mapping scheme}\footnote{The definition is originally  introduced by Milnor,  see \cite{M}.}.   
 There are three types of mapping schemes, up to M\"obius conjugacy:

 \textbf{Type I  :  $f(D_0)=D_0,  f(D_\infty)=D_0.$ }    In this case,
  $D_0$ contains an attracting  fixed point of $f$  and the number $n$ is odd. 
  We  normalize $f$ so that $f(0)=0$, $1\in \partial D_0$, $f(1)=1$,  and $\infty$ is the {\it conformal barycenter} \footnote{Let $\Omega$ be a Riemann surface isomorphic to $\mathbb{D}$, the {\it  conformal barycenter} of  the points $p_1, \cdots, p_k \in \Omega$ is $\phi^{-1}(0)$, where $\phi:  \Omega\rightarrow \mathbb{D}$ is the unique Riemann mapping satisfying that $\phi(p_1)+\cdots+ \phi(p_k)=0$.  See \cite{M}.}   of $f^{-1}(0)$ in 
$D_\infty$, 
counted with multiplicity.  The mapping scheme is as follows

\begin{center}       \begin{tikzpicture}    [->,>=stealth',shorten >=3pt,auto,node distance=1.0cm,
                    semithick]              
  \tikzstyle{every state}=[fill=none,draw=none,text=black]     
              \node[state]         (D1r) at (3, 1.2)              {$D_1$};
                   \node[state]         (D3r) at (3, 0.4)              {$D_3$};
    \node[state]         (dotsr) at (3, -0.4)              {$\cdots$};
     \node[state]         (Dn1) at (3, -1.2)              {$D_{n}$};
        \node[state]         (Dinftyr) at (5, 0)              {$D_\infty$};
          \node[state]         (D0r) at (7, 0)              {$D_0$ };

          \node[state]         (D2r) at (9, 1.2)              {$D_2$};
                   \node[state]         (D4r) at (9, 0.4)              {$D_4$};
    \node[state]         (dotsrr) at (9, -0.4)              {$\cdots$};
     \node[state]         (Dn) at (9, -1.2)              {$D_{n-1}$};
      
         \path    (D1r) edge[bend left=0]              node {} (Dinftyr)
               (D3r) edge[bend left=0]              node {} (Dinftyr)
               (dotsr) edge[bend left=0]              node {} (Dinftyr)
               (Dn1) edge[bend left=0]              node {} (Dinftyr)
              (Dinftyr) edge[bend left=0]              node {} (D0r)
              (D2r) edge[bend left=0]              node {} (D0r)
               (D4r) edge[bend left=0]              node {} (D0r)
               (dotsrr) edge[bend left=0]              node {} (D0r)
             (Dn) edge[bend left=0]              node {} (D0r);
              \def\circledarrow#1#2#3{ 
\draw[#1,->] (#2) +(105:#3) arc(240:-100:#3) ;}
\circledarrow{black}{D0r}{0.3cm};
               \end{tikzpicture}
        \end{center}

%
%
%
%

    \textbf{Type II :  $f(D_0)=D_\infty, f(D_\infty)=D_0$. } In this case,  the number $n$ is even and
 $f$ has an attracting cycle of period two. We may normalize $f$ so that $f(0)=\infty$, $f(\infty)=0$, $1\in \partial D_0$ and $f^2(1)=1$.
 The mapping scheme is

 \begin{center}       \begin{tikzpicture}    [->,>=stealth',shorten >=3pt,auto,node distance=1.0cm,
                    semithick]              
  \tikzstyle{every state}=[fill=none,draw=none,text=black]     
              \node[state]         (D1r) at (3, 1.2)              {$D_1$};
                   \node[state]         (D3r) at (3, 0.4)              {$D_3$};
    \node[state]         (dotsr) at (3, -0.4)              {$\cdots$};
     \node[state]         (Dn1) at (3, -1.2)              {$D_{n-1}$};
        \node[state]         (Dinftyr) at (5, 0)              {$D_0$};
          \node[state]         (D0r) at (7, 0)              {$D_\infty$ };

          \node[state]         (D2r) at (9, 1.2)              {$D_2$};
                   \node[state]         (D4r) at (9, 0.4)              {$D_4$};
    \node[state]         (dotsrr) at (9, -0.4)              {$\cdots$};
     \node[state]         (Dn) at (9, -1.2)              {$D_n$};
      
         \path    (D1r) edge[bend left=0]              node {} (Dinftyr)
               (D3r) edge[bend left=0]              node {} (Dinftyr)
               (dotsr) edge[bend left=0]              node {} (Dinftyr)
               (Dn1) edge[bend left=0]              node {} (Dinftyr)
               (Dinftyr) edge[bend left=40]              node {} (D0r)
             (D0r) edge[bend left=40]              node {} (Dinftyr)
               (D2r) edge[bend left=0]              node {} (D0r)
               (D4r) edge[bend left=0]              node {} (D0r)
               (dotsrr) edge[bend left=0]              node {} (D0r)
             (Dn) edge[bend left=0]              node {} (D0r);
     \end{tikzpicture}
        \end{center}   
        

\textbf{Type III :     $f(D_0)=D_0, f(D_\infty)=D_\infty$.}  In this case,  the number $n$ is even, and
    each of $D_0$ and $D_\infty$ contains an attracting fixed point of $f$.  We may normalize $f$ so that $f(0)=0, f(\infty)=\infty$, $1\in \partial D_0$ and $f(1)=1$.
    The mapping scheme is

\begin{center}
\begin{tikzpicture}[->,>=stealth',shorten >=3pt,auto,node distance=1.0cm,
                    semithick]              
  \tikzstyle{every state}=[fill=none,draw=none,text=black]
     \node[state]         (D1r) at (3, 1.2)              {$D_1$};
                   \node[state]         (D3r) at (3, 0.4)              {$D_3$};
    \node[state]         (dotsr) at (3, -0.4)              {$\cdots$};
     \node[state]         (Dn1) at (3, -1.2)              {$D_{n-1}$};
        \node[state]         (Dinftyr) at (5, 0)              {$D_\infty$};
          \node[state]         (D0r) at (10, 0)              {$D_0$ };
                 
          \node[state]         (D2r) at (8, 1.2)              {$D_2$};
                   \node[state]         (D4r) at (8, 0.4)              {$D_4$};
    \node[state]         (dotsrr) at (8, -0.4)              {$\cdots$};
     \node[state]         (Dn) at (8, -1.2)              {$D_n$};
      
         \path    (D1r) edge[bend left=0]              node {} (Dinftyr)
               (D3r) edge[bend left=0]              node {} (Dinftyr)
               (dotsr) edge[bend left=0]              node {} (Dinftyr)
               (Dn1) edge[bend left=0]              node {} (Dinftyr)
               (D2r) edge[bend left=0]              node {} (D0r)
               (D4r) edge[bend left=0]              node {} (D0r)
               (dotsrr) edge[bend left=0]              node {} (D0r)
             (Dn) edge[bend left=0]              node {} (D0r);

 \def\circledarrow#1#2#3{ 
\draw[#1,->] (#2) +(25:#3) arc(160:-160:#3) ;}
\circledarrow{black}{Dinftyr}{0.5cm};

 \def\circledarrow#1#2#3{ 
\draw[#1,->] (#2)+(25:#3) arc(160:-160:#3) ;}
\circledarrow{black}{D0r}{0.5cm};
 \end{tikzpicture}
 \end{center}

In either case,  set $d_k=\deg(f|_{A_k})$. It's clear that $d=\sum_{k=1}^{n+1}  d_k$. Note that the annuli $A_k$'s are contained in $A$ and  they have disjoint closures, by   the Gr\"otzsch inequality
 $${\rm mod}(A)> \sum_{k=1}^{n+1} {\rm mod}(A_k)={\rm mod}(A)\sum_{k=1}^{n+1}\frac{1}{d_k} \    \Longrightarrow \   \sum_{k=1}^{n+1}\frac{1}{d_k}<1.$$

Therefore, the dynamics of  $f$ induces a partition of   $d$ with number vector $\boldsymbol{d}=\boldsymbol{d}(f)=(d_1, d_2,\cdots,d_{n+1})$ satisfying 
$$|\boldsymbol{d}|:=\sum_{k=1}^{n+1} d_k=d,   \ \  \   \sum_{k=1}^{n+1}\frac{1}{d_k}<1.  \ \   (*)$$
We say a number vector $\boldsymbol{d}\in \mathbb{N}^{n+1}$ satisfying $(*)$ is {\it admissible}. 

Note  also that $n\geq 1$. By $(*)$,  all $d_k\geq 2$  and at most one of $d_k$ is two, therefore $d\geq 2+3=5$. More generally,  by the mean inequality,
$$d>\Big(\sum_{k=1}^{n+1}d_k\Big)\Big(\sum_{k=1}^{n+1}\frac{1}{d_k}\Big)\geq (n+1)^2 \     \Longrightarrow n< \sqrt{d}-1.$$

We remark that the mapping scheme can be  recorded by the pair $(\mathcal{I}, \tau)$, where  $\mathcal{I}=\{0, \infty,1,\cdots, n\}$ is the index set of the Fatou components in $\boldsymbol{F}(f)$,  and  $\tau$ is a self map of $\mathcal{I}$ defined  by  $\tau(k)=j$ if $f(D_k)=D_j$.

\subsection{Boundary marking and the space $\mathcal{F}^{\chi_0}_{\sigma,\boldsymbol{d}}$} \label{boundary-marking} From the previous subsection, we see that  the hyperbolic component $\mathcal{H}\subset M_d$, consisting of Type-$\sigma$ ($\sigma\in \{I, II, III\}$)  maps,  
 induces an integer  $n\geq1$ (the number of critical annular Fatou components),  and an admissible partition vertor $\boldsymbol{d}=(d_1, d_2,\cdots,d_{n+1})\in  \mathbb{N}^{n+1}$.

Let $\mathcal{F}_{\sigma, \boldsymbol{d}}\subset {\rm Rat}_d$ be the collection of   Type-$\sigma$ hyperbolic rational maps $f$ whose Julia sets are Cantor circles, normalized as the previous subsection, and $\boldsymbol{d}(f)=\boldsymbol{d}$.
The set $\mathcal{F}_{\sigma, \boldsymbol{d}}$,  as a subspace of ${\rm Rat}_d$,    might be disconnected.

Following Milnor \cite{M},  a {\it boundary marking} for a map $f\in  \mathcal{F}_{\sigma, \boldsymbol{d}}$ means a function     $\nu:  \boldsymbol{F}(f)\rightarrow \mathbb{C}$ which assigns to each  $U\in \boldsymbol{F}(f)$ a  boundary point $\nu(U)$, satisfying that
$$f(\nu(U))=\nu(f(U)),   \  \nu(D_0)=1.$$

Note that the choice of the boundary marking is not unique. In fact, when we fix the point $\nu(f(U))$, there are at most $\deg(f|_{U})$ choices of the  marking point  $\nu(U)$.
It follows that there are finitely many  choices of the boundary marking  $\nu$.  
The {\it characteristic} of  $\nu$, denoted by $\chi(\nu)$, is a symbol vector $\chi=(\epsilon_1, \cdots, \epsilon_n)\in \{\pm\}^n$, defined in the way that
$$\epsilon_k= \begin{cases} +, &\text{ if } \nu(D_k)\in \partial D_k \cap \partial A_{k+1}, \\
-, &\text{ if } \nu(D_k)\in \partial D_k \cap \partial A_{k}.
\end{cases}$$

We call the pair $(f, \nu)$ a {\it marked map}. Fix a symbol vector $\chi_0$,
define $$\mathcal{F}^{\chi_0}_{\sigma,\boldsymbol{d}}=\{(f, \nu); f\in  \mathcal{F}_{\sigma,\boldsymbol{d}},  \
\chi(\nu)=\chi_0 \}.$$

The set $\mathcal{F}^{\chi_0}_{\sigma,\boldsymbol{d}}$ has a natural topology so that every map of $\mathcal{F}_{\sigma,\boldsymbol{d}}$ has a
neighborhood $N$ which is evenly covered  under the projection $\mathcal{F}^{\chi_0}_{\sigma,\boldsymbol{d}}\rightarrow\mathcal{F}_{\sigma,\boldsymbol{d}}$, defined by sending $(f, \nu)$ to $f$.  To see this, it is enough to note that  each marked point  $\nu(D_k)$ of $f$ is preperiodic and
eventually repelling, and therefore deforms continuously as we deform the map $f$.
 As a topology space, $\mathcal{F}^{\chi_0}_{\sigma,\boldsymbol{d}}$ might be disconnected.

\subsection{Model map} \label{model-map0}

For each marked map 
$(f, \nu)\in \mathcal{F}^{\chi_0}_{\sigma,\boldsymbol{d}}$ and each Fatou component $D_j\in \boldsymbol{F}(f)$, there is a conformal isomorphism 
$\kappa_{j}$ mapping $D_j$ onto $\mathbb{D}$ or $\mathbb{A}_{r_j} $ (here $r_j=e^{-2\pi {\rm mod}(D_j)}$ if $D_j$ is an annulus), satisfying that $\kappa_{j}(\nu(D_j))=1$. 
If $D_j$ is a disk, then the index  $j$ is either $0$ or $\infty$, in this case we further require that $\kappa_j(j)=0\in \mathbb{D}$.

We then get a {\it model map} $m_j$ for $f|_{D_j}$,  defined so that the  following  diagram is commutative:
$$
\xymatrix{ & D_j \ar[r]^{f|_{D_j}  \ \ \ \   }
\ar[d]_{\kappa_j}
& D_{\tau(j)}  \ar[d]^{\kappa_{\tau(j)}  }\\
&\mathbb{D} \text{ or } \mathbb{A}_{r_j}  \ar[r]_{m_j} & \mathbb{D}}
$$ 

Now let's look at the model maps $m_j$.  First, it is a  standard fact that any proper holomorphic map $\beta$ from $\mathbb{D}$ onto itself with $\beta(1)=1$  can be written uniquely as a $D$-fold Blaschke
product
$$ \beta(z)= \Bigg( \prod_{k=1}^{D}\frac{1-\bar{a}_k}{1-a_k}\Bigg) \left(\prod_{k=1}^{D} \frac{z-{a}_k} {1-\bar{a}_k z}\right),  \ \  a_1, \cdots, a_{D} \in \mathbb{D}.$$

%
%
  We say $\beta$ is {\it fixed point centered} if $\beta(0)=0$, and {\it zero centered} if $0$ is the conformal barycenter  of $\beta^{-1}(0)$ (counted with multiplicity) in $\mathbb{D}$. 
  
Denote by  $\boldsymbol{B}^{\rm{fc}}_D$ the  space of  degree-$D$  fixed point centered  Blaschke products fixing the boundary point $1$;
   by   $\boldsymbol{B}^{\rm{zc}}_D$, the space of  degree-$D$   zero centered   Blaschke products fixing the boundary point $1$.
  


It's clear that $m_{0}\in \boldsymbol{B}^{\rm{fc}}_{d_1}$. If $f\in  \mathcal{F}_{I, \boldsymbol{d}}$, then the model map  $m_{\infty}\in \boldsymbol{B}^{\rm{zc}}_{d_{n+1}}$; if  $f\in  \mathcal{F}_{II,\boldsymbol{d}}$ or 
 $ \mathcal{F}_{III, \boldsymbol{d}}$,  then $m_{\infty}\in\boldsymbol{B}^{\rm{fc}}_{d_{n+1}}$.
 
For $1\leq k\leq n$,  the model map $m_{k}$ is a proper holomorphic map  from the annulus  $\mathbb{A}_{r_k}$ onto $\mathbb{D}$,  fixing $1$ and with mapping degree $d_k+d_{k+1}$.  Let  $\delta_k\in\{d_k, d_{k+1}\}$ be the  mapping degree of $m_k$ on the inner boundary $C_{r_k}$ of $\mathbb{A}_{r_k}$.  Clearly, $\delta_k=d_{k+1}$ if $\epsilon_k=-$, and $\delta_k=d_{k}$ 
if $\epsilon_k=+$.

\begin{defi}[Model space of annulus-disk mappings]\label{modelAD} Given integers $e>\delta>0$,  let $\boldsymbol{A}(e, \delta)$ be the set of all annulus-disk mappings, defined by
$$\boldsymbol{A}(e, \delta)=\left\{
\begin{array}{c}
 m: \mathbb{A}_r \rightarrow \mathbb{D} \text{ is a proper and holomorphic map} ; \\
 {\rm deg}(m)=e, \ {\rm deg}(m|_{C_r})=\delta, \  m(1)=1, \  r\in (0,1)
\end{array} 
 \right\}.$$
 
The topology of  $\boldsymbol{A}(e, \delta)$ is given as follows: We say that the maps $\phi_k: \mathbb{A}_{r_k}\rightarrow \mathbb{D}$  converges to  $\phi: \mathbb{A}_{r}\rightarrow \mathbb{D}$ in $\boldsymbol{A}(e, \delta)$, if $r_k\rightarrow r$, and  for any compact subset $K\subset \mathbb{A}_{r}$, $\phi_k|_K$ converges  uniformly
 to $\phi|_K$ for $k$ sufficiently large.
 \end{defi}

It's easy to see that the map $m_k: \mathbb{A}_{r_k}\rightarrow\mathbb{D}$ obtained above is an element of $ \boldsymbol{A}(d_k+d_{k+1},  \delta_k)$.
Now, we define the model space $\mathbf{M}_\sigma$ by
$$\mathbf{M}_\sigma=
\begin{cases}
\boldsymbol{B}^{\rm{fc}}_{d_1}\times \boldsymbol{B}^{\rm{zc}}_{d_{n+1}}\times \prod_{k=1}^n \boldsymbol{A}(d_k+d_{k+1},  \delta_k), &\text{ if } \sigma=I,  \\
\boldsymbol{B}^{\rm{fc}}_{d_1}\times \boldsymbol{B}^{\rm{fc}}_{d_{n+1}}\times \prod_{k=1}^n \boldsymbol{A}(d_k+d_{k+1},  \delta_k), &\text{ if } \sigma={II} \text{ or } {III}.
\end{cases}$$
 
 There is a natural map from  $\mathcal{F}^{\chi_0}_{\sigma,\boldsymbol{d}}$ to  $\mathbf{M}_\sigma$
$$\rho:  \mathcal{F}^{\chi_0}_{\sigma,\boldsymbol{d}} \rightarrow \mathbf{M}_\sigma, \ \ \ \  (f,\nu)\mapsto \boldsymbol{m}(f)=(m_{0}, m_{{\infty}}, m_{1}, \cdots, m_{{n}}).$$

Two questions naturally arise:

\begin{que}\label{question21} What is the topology of the model spaces   $$\boldsymbol{B}^{\rm{fc}}_{D}, \ \boldsymbol{B}^{\rm{zc}}_{D}, \ \boldsymbol{A}(d_k+d_{k+1},  \delta_k)?$$
 \end{que}

\begin{que}\label{question22} What is the property of $\rho$? Can we know the topology of the marked hyperbolic component (and further  $\mathcal{H}$) from that of $\mathbf{M}_\sigma$?
 \end{que}

For the first question, the following is known 

\begin{lem} [Lemma 4.9 \cite{M}]  \  \label{model-space0} For any integer $D\geq2$, the  model spaces  $\boldsymbol{B}_D^{\rm{fc}}, \boldsymbol{B}_D^{\rm{zc}}$ both are homeomorphic to $\mathbb{R}^{2D-2}$.
\end{lem}

So the essential difficulty for the first question is to characterize the topology of  the space $\boldsymbol{A}(e, \delta)$. This  will be done in Sections \ref{muti-disk}, \ref{annu-disk} and \ref{model}.

For the second  question,  we will prove the finite  covering property  of $\rho$ in Sections \ref{cov} and \ref{finite} and further use this property to describe the topology of  the marked hyperbolic component   in Section \ref{top_marked} and hyperbolic component in Section \ref{global-hyp}.


\section{Proper mapping from multi-connected domain to disk} \label{muti-disk}

In this section, we  consider a slightly  general  question:  Given   a multi-connected domain $\Omega$ and a  finite set $Z\subset \Omega$ ,   under what conditions  there is a  proper holomorphic  mapping from $\Omega$ onto  the unit disk 
$\mathbb{D}$, with $Z$ as the prescribed zero set? 
 Here, we recall that a continuous map $g: X\rightarrow Y$ between two topological spaces 
  is said {\it proper},
 if the preimage $g^{-1}(K)$ of every compact subset $K$ of $Y$ is compact in $X$.

To make the question precise,  let $\Omega_g$ be a  multi-connected planar domain  bounded by  the Jordan curves $\gamma_0, \gamma_1, \cdots, \gamma_g$, where $g\geq 1$ is an integer.  We associate each boundary curve $\gamma_k$ with  a 
positive integer $d_k$.   Let $p_1, \cdots, p_e$ be $e=\sum_{k=0}^g d_k$ points in $\Omega_g$, not necessarily distinct.
 
 \begin{que}\label{question1} Is there a proper holomorphic map   $f: \Omega_g \rightarrow \mathbb{D}$ of degree 
$e$ with $f^{-1}(0)=\{p_1, \cdots, p_e\}$ and ${\deg (f|_{\gamma_k})}=d_k$ for all $0\leq k\leq g$?
 \end{que}
 
 We remark that the notation `$\{\}$' here and throughout the paper should be understood as  the  `weighted set', or the  {\it divisor}. In other words,  
 when two (or more) points are same, there is a multiplicity  serving as the weight.  For example, $\{p,p,q\}=\{2p,q\}\neq\{p,q\}$.
 
In general, the answer to Question \ref{question1} is negative.  The aim of this section is to give a necessary and sufficient condition to guarantee the existence of  the proper map,
using  Abel's Theorem for principle divisors in the Riemann surface theory.
 The main result, with an independent interest,  is as follows:  


\begin{thm}\label{proper-m-d}    The following two statements are equivalent:

1. There is a proper holomorphic map   $f: \Omega_g\rightarrow \mathbb{D}$ of degree 
$e=\sum_{k=0}^g d_k$ with $f^{-1}(0)=\{p_1, \cdots, p_e\}$ and ${\deg (f|_{\gamma_k})}=d_k$  for all $0\leq k\leq g$.

2. The following equations hold
$$\sum_{j=1}^e u_{k}(p_j)=d_k, \  \    1\leq k\leq g,$$
where $ u_{k}: \Omega_g\rightarrow \mathbb{R}$  is a harmonic function  satisfying that 
$$ u_{k}|_{\gamma_k}=1, \   u_{k}|_{\partial\Omega_g\setminus\gamma_k}=0.$$

Moreover, the proper holomorphic map   $f$ is unique if we specify  the value of $f$ at a boundary point $q\in \partial \Omega_g$. 
\end{thm}

\begin{proof}  By  Koebe's generalized Riemann mapping  theorem \cite{Ko1, Ko2}, $\Omega_g$ is bi-holomorphic to   a multi-connected domain  whose  boundaries are  round circles. 
For this, we may assume that  all $\gamma_k$ are round circles.

 We consider the Schottky double surface $S_g=\Omega_g\cup \partial \Omega_g \cup \Omega_g^*$ of $\Omega_g$, 
which is a compact Riemann surface of genus $g$.   It admits an anti-holomorphic involution $\sigma$,  fixing 
$\partial \Omega_g$ and
mapping $\Omega_g$ to its symmetric part $\Omega_g^*$. 
A basis of the homology group $H_1(S_g)$ can be chosen as
 $\gamma_1,\cdots, \gamma_g$ and $\alpha_1, \cdots, \alpha_g$, see Figure 3.
Let $\omega_1, \cdots, \omega_g$ be the dual basis of $\gamma_1,\cdots, \gamma_g$ in $\mathcal{H}^1(S_g)$, the space of holomorphic differentials
on $S_g$, satisfying that 
$$\int_{\gamma_j} \omega_k= \delta_{kj}, \ \ j,k=1,\cdots,g.$$
\begin{figure}[h] 
\begin{center}
\includegraphics[height=6cm]{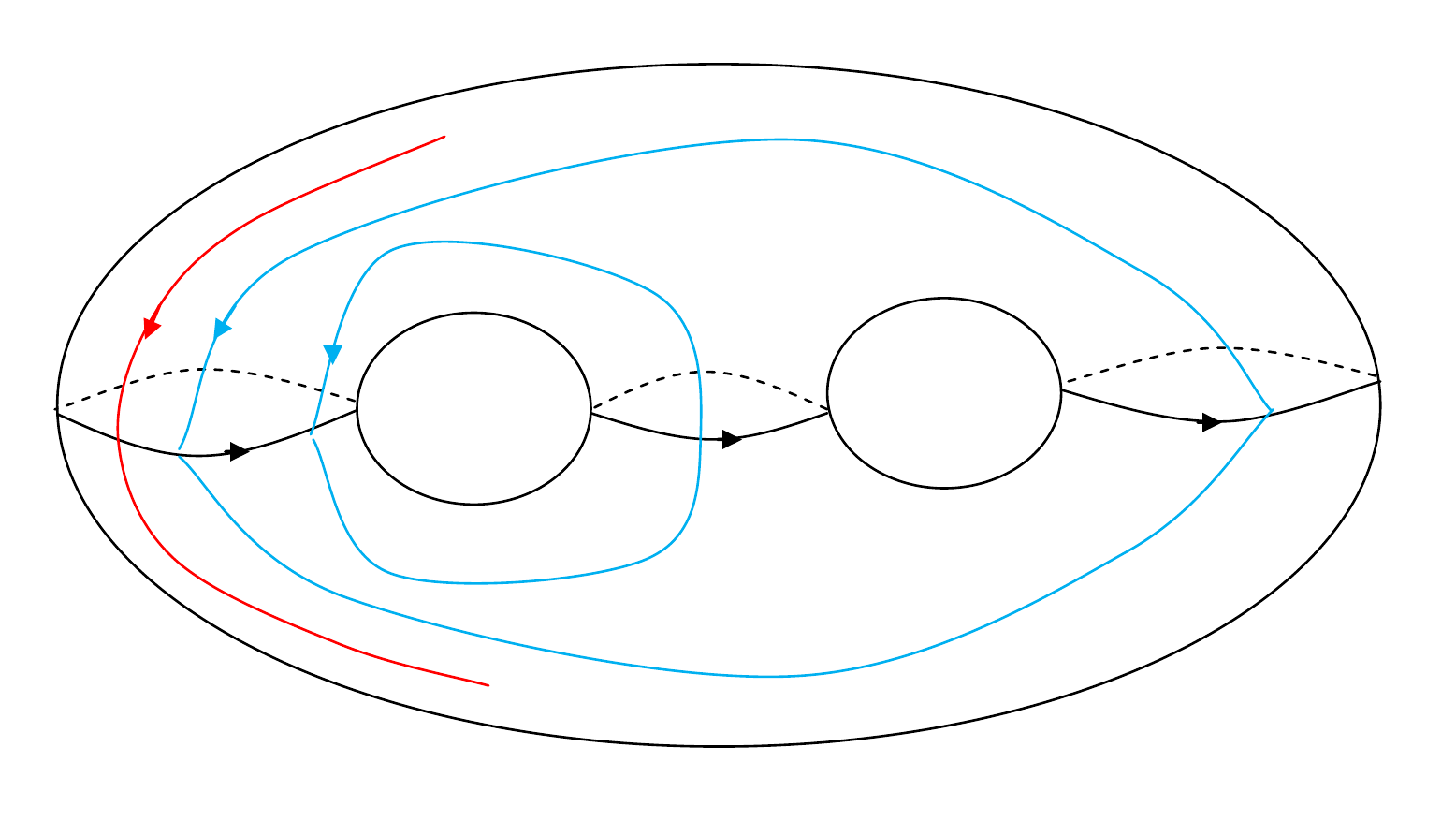}
 \put(-286,98){$\beta_z$} \put(-255,85){$\gamma_0$} \put(-175,72){$\gamma_1$} \put(-72,76){$\gamma_2$} \put(-145,50){$\Omega_g$} \put(-145,120){$\Omega_g^*$}
  \put(-185,105){$\alpha_1$}   \put(-185,130){$\alpha_2$}  \put(-215,138){$\bullet  \ z^*$}  \put(-205,26){$\bullet  \ z$}
\caption{The  Schottky double surface $S_g$ of $\Omega_g$, when $g=2$}
\end{center}
\end{figure}
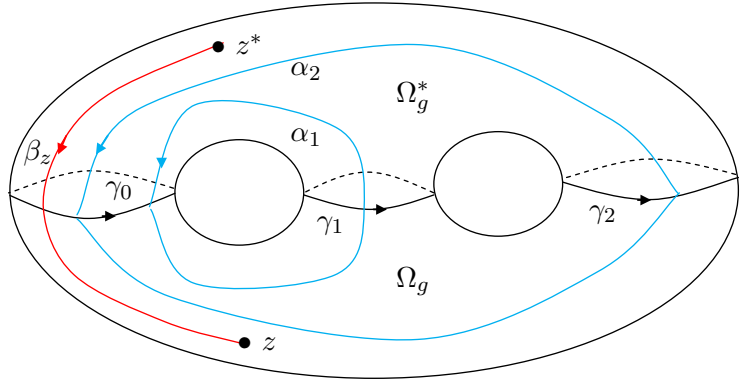

By replacing $\omega_k$ with $\frac{1}{2}(\omega_k+ \overline{\sigma^*\omega_k})$, 
  we may assume that   $\sigma^*\omega_k= \overline{\omega}_k$.
Let 
$$b_{kj}=\int_{\alpha_j} \omega_k,   1\leq k,j\leq g.$$
It's known that $B=(b_{kj})$  satisfies $\overline{B}= -B$ (this means the real part of $B$ vanishes) and the imaginary  part of $B$ is symmetric and positive definite  (see \cite[Proposition, p.63]{FK},   note that  here we chose  a different  
 basis of $H_1(S_g)$ from that in \cite{FK}). 
For any $z\in \Omega_g$, let $\beta_z$ be a curve in $S_g$ connecting  $z^*=\sigma(z)$ to $z$, symmetric about $\partial\Omega_g$. Define a function 
$$\Phi_k(z)=\int_{\beta_z} \omega_k.$$
It's easy to check that  $z\mapsto\Phi_k(z)$ is a  harmonic function on $\Omega_g$,  with vanishing real part, and satisfying  that
$$\Phi_k|_{\gamma_0}=0  \text{\  and \ } \Phi_k|_{\gamma_j}=\int_{\alpha_j} \omega_k=b_{kj}, \  1\leq j\leq g.$$
Therefore, we have 
$$ \Phi_k=\sum_{j=1}^g b_{kj} u_{{j}}.$$
By Abel Theorem (see \cite[Theorem 7.26]{S} or \cite[Theorem, p.93]{FK}),   there is a proper holomorphic map $F: S_g\rightarrow \mathbb{\widehat{C}}$ for which  $F^{-1}(0)=\{p_1, \cdots, p_e\}$ and  $F^{-1}(\infty)=\{p^*_1, \cdots, p^*_e\}$  (which is  equivalent to the statement: there is a proper holomorphic map   $f: \Omega_g\rightarrow \mathbb{D}$ of degree $e$ with $f^{-1}(0)=\{p_1, \cdots, p_e\}$),  if and only if

$$\left(
\begin{array} {c}
\sum_{k=1}^e\Phi_1(p_k)\\
\sum_{k=1}^e\Phi_2(p_k)\\
\vdots \\ 
\sum_{k=1}^e\Phi_g(p_k)\end{array}
\right)\in \mathbb{Z}\left(
\begin{array} {c}
b_{11}\\
b_{21}\\
\vdots \\ 
b_{g1}\end{array}
\right)+\cdots+\mathbb{Z}\left( \begin{array} {c}
b_{1g}\\
b_{2g}\\
\vdots \\ 
b_{gg}\end{array}
\right).$$
 So  there are integers $n_1, \cdots, n_g$ such that
 $$\left(
\begin{array} {c}
\sum_{k=1}^e\Phi_1(p_k)\\
\sum_{k=1}^e\Phi_2(p_k)\\
\vdots \\ 
\sum_{k=1}^e\Phi_g(p_k)\end{array}
\right)
=B \left(
\begin{array} {c}
\sum_{k=1}^e u_1(p_k)\\
\sum_{k=1}^e u_2(p_k)\\
\vdots \\ 
\sum_{k=1}^e u_g(p_k)\end{array}
\right)= B\left( \begin{array} {c}
n_{1}\\
n_{2}\\
\vdots \\ 
n_{g}\end{array}
\right).$$
Note that the matrix  $B$ is  reversible, we have
  $$\sum_{j=1}^e u_{k}(p_j)=n_k, \  \    1\leq k\leq g.$$
  These equalities guarantee the existence of the proper map $f: \Omega_g\rightarrow \mathbb{D}$.
  To determine the integers $n_1, \cdots, n_g$,
  we define a  vector-valued function  $\Psi: \mathbb{D}\rightarrow \mathbb{R}^g$ by
  $$\Psi(\zeta)=\Big(\sum_{z\in f^{-1}(\zeta)}u_1(z), \cdots, \sum_{z\in f^{-1}(\zeta)}u_g(z)\Big)^t.$$
  It's clear that $\Psi$ is continuous. By above argument,  $\Psi$ takes discrete values in $\mathbb{Z}^g$.
So it is a constant vector.  Note that we have required that ${\deg (f|_{\gamma_k})}=d_k$ for all $0\leq k\leq g$, meaning that when $\zeta$ approaches 
$\partial \mathbb{D}$,  the value $ \Psi(\zeta)$ will approach the constant vector $(d_1, \cdots, d_g)^t$. Therefore we have
    $$\sum_{j=1}^e u_{k}(p_j)=d_k, \  \    1\leq k\leq g.$$
 All the arguments above are reversible, implying the equivalence. 
 
 To finish, we prove the uniqueness part.   Any  proper holomorphic  map $f: \Omega_g\rightarrow \mathbb{D}$ can induce a holomorphic map $\widehat{f}: S_g \rightarrow \mathbb{\widehat{C}}$ by reflection, with zeros $p_1, \cdots, p_e$ and poles    $p^*_1, \cdots, p^*_e$. Let  $f_1, f_2$ be two  proper holomorphic  maps satisfying the first statement of the theorem, then $\widehat{f_2}/ \widehat{f_1}: S_g\rightarrow \mathbb{{C}}$ is a holomorphic map. Therefore $\widehat{f_2}/ \widehat{f_1}$ is  necessarily  constant. Since they are identical  at the point $q\in \partial \Omega_g$, we have
 $\widehat{f_2}= \widehat{f_1}$, equivalently, $f_2=f_1$. 
 \end{proof}




\section{Proper mapping from annulus to disk} \label{annu-disk}

We now focus on a special case of Theorem \ref{proper-m-d}, that is, $\Omega_g$ is an annulus. 
Recall that  $r\in(0,1)$, $\mathbb{A}_r=\{z\in \mathbb{C};  r<|z|<1\}$ and $C_r=\{|z|=r\}$ be the inner boundary of $\mathbb{A}_r$.  Let $p_1, p_2, \cdots,  p_e$ be $e\geq 2$ points in 
$\mathbb{A}_r$, not necessarily distinct.  Let $\delta\in[1,e)$ be an integer.

\begin{thm}\label{proper-a-d}  There is a  proper holomorphic map   $f: \mathbb{A}_r \rightarrow \mathbb{D}$ of degree 
$e$ with $f^{-1}(0)=\{p_1, \cdots, p_e\}$ and ${\deg}(f|_{C_r})=\delta$  if and only if 
$$|p_1p_2\cdots p_e|=r^{\delta}.$$ 
When exists, the proper map $f$ is unique if we further require that $f(1)=1$.
\end{thm}

\begin{proof}  By Theorem \ref{proper-m-d}, the  proper holomorphic map as required exists if and only if 
$$u(p_1)+\cdots+u(p_e)=\delta,$$
where $u$ is a harmonic function on $\mathbb{A}_r$ with $u|_{C_r}=1$ and $u|_{\partial{\mathbb{D}}}=0$. It's  easy to observe that 
$u(z)=\log |z|/\log r$.  This gives the required  equivalent condition $|p_1p_2\cdots p_e|=r^{\delta}.$ The uniqueness also follows from  Theorem \ref{proper-m-d}.
\end{proof}

\begin{rmk}  An alternative proof of the `only if' part of Theorem \ref{proper-a-d} goes as follows:  we define 
$$g(w)=\prod_{f(z)=w}z,  \ w\in \mathbb{D}$$
where the product is taken counted multiplicity.    It's clear that $g$ is continuous and non-vanishing on $\mathbb{D}$, holomorphic except at critical values. By the removable singularity theorem,  $g:\mathbb{D}\rightarrow\mathbb{C}^*$ is holomorphic. Applying the maximum principle to $g$ and $1/g$,  we get
$$|g(w)|\leq \max_{\zeta\in \partial \mathbb{D}}|g(\zeta)|, \  |1/g(w)|\leq \max_{\zeta\in \partial \mathbb{D}}|1/g(\zeta)|.$$
Since $g$ is proper, we have  $|g(\zeta)|=r^\delta$ for all $\zeta\in  \partial \mathbb{D}$.  Therefore $g$ is a constant map and $|g(w)|=r^\delta$ for all $w\in\mathbb{D}$. In particular, $|g(0)|=|p_1p_2\cdots p_e|=r^\delta$.
\end{rmk}

\begin{thm}\label{model-map} A  proper holomorphic map   $f: \mathbb{A}_r \rightarrow \mathbb{D}$ of degree 
$e$ with 
$$f^{-1}(0)=\{p_1, \cdots, p_e\}, \ f(1)=1 \text{\ and }  {\deg}(f|_{C_r})=\delta$$ can be written uniquely as
$$f(z)= {z^{-\delta}} B_{0}(z)\prod_{j\geq 1} \Big(B_{j}(z)B_{-j}(z)\Big),  \  z\in  \mathbb{A}_r ,$$
where 
$$B_{j}(z)= \Bigg(\prod_{k=1}^e\frac{1-\overline{p_k}r^{2j}}{1-{p_k} r^{2j}}\Bigg) \Bigg(\prod_{k=1}^e  \frac{z-p_k r^{2j}}{1-\overline{p_k} r^{2j}z} \Bigg), \ j\in\mathbb{Z}.$$
\end{thm} 
\begin{proof} Write 
$$f_N(z)={z^{-\delta}} B_{0}(z)\prod_{j=1}^N \Big(B_{j}(z)B_{-j}(z)\Big), N\in \mathbb{N}\cup \{\infty\}.$$
In order to show $f=f_\infty$, we need to verify that $f_\infty|_{\mathbb{A}_r}$ is a proper holomorphic map
from $\mathbb{A}_r$ onto $\mathbb{D}$
of degree $e$, fixing $1$ and having the same zero set and the boundary degree  as $f$. Note that by Theorem \ref{proper-a-d}, we know that $|p_1\cdots p_e|=r^\delta$. The proof proceeds in four steps:

  \vspace{6 pt} 

\textbf{Step 1. } {\it  $f_N$ converges locally and uniformly to $f_{\infty}$ on $\mathbb{C}^*$, as $N\rightarrow \infty$.} 

\vspace{6 pt} 

We first show that $f_N$ converges   uniformly to $f_{\infty}$ on $\overline{\mathbb{A}}_r$.
In fact, it is not hard to see that there is an integer $M=M(r)>0$ and a constant $c=c(r)>0$, such
that when $N\geq M$,
$$|B_{N}(z)B_{-N}(z)-1|\leq c r^{2N},  \  \forall \  z\in \overline{\mathbb{A}}_r.$$
Therefore 
$$|f_\infty(z)-f_N(z)|= |f_N(z)||1-\prod_{j=N+1}^\infty \big(B_{j}(z)B_{-j}(z)\big)|\leq C r^{2N},$$
where $C$ is a constant, dependent  only on $r$. The uniform convergence on $\overline{\mathbb{A}}_r$ follows immediately. 
With the same argument, one can show that $f_N$ converges  uniformly to $f_{\infty}$ on $r^{2k}\overline{\mathbb{A}}_r$  for any $k\in\mathbb{Z}$.

Note that $\mathbb{C}^*=\bigcup_{k\in\mathbb{Z}}r^{2k}(\overline{\mathbb{A}_r\cup\mathbb{A}^*_r})$, where $\mathbb{A}^*_r=
\{z\in\mathbb{C}; 1/\bar{z}\in \mathbb{A}_r\}$, and that for any $N$, the map $f_N$ satisfies  $1/\overline{f_N(z)}= f_N(1/\bar{z})$. Therefore $f_N: \mathbb{C}^*\rightarrow \mathbb{\widehat{C}}$ converges locally and uniformly, in the spherical metric,  to $f_\infty:\mathbb{C}^*\rightarrow \mathbb{\widehat{C}}$.
Moreover, the map $f_{\infty}$ also satisfies $1/\overline{f_\infty(z)}= f_\infty(1/\bar{z})$.

  \vspace{6 pt} 

\textbf{Step 2. }   {\it  $f_{\infty}: \mathbb{C}^*\rightarrow \mathbb{\widehat{C}}$ is modular:   $f_\infty(r^2z)=f_\infty(z)$.} 

  \vspace{6 pt} 

  Fix  a compact subset $K\subset \mathbb{C}^*$,    one may verify that for any $z\in K$, 
\bess \frac{f_N(r^2z)}{f_N(z)}&=&\frac{1}{r^{2\delta}} \prod_{k=1}^e  \frac{(z-{p_k} r^{-2(N+1)}) (z-\overline{p_k}^{-1} r^{2N})}
{(z-\overline{p_k}^{-1} r^{-2(N+1)}) (z-{p_k} r^{2N})}\\
&\rightarrow&  \frac{|p_1\cdots p_e|^2}{r^{2\delta}} =1 \text{ as } N\rightarrow \infty. \eess
By Step 1, we have $f_\infty(r^2z)=f_\infty(z)$ for all $z\in K$. By the identity theorem for holomorphic maps, the equality holds for all $z\in \mathbb{C}^*$.

  \vspace{6 pt} 

\textbf{Step 3. }   {\it  $f_{\infty}|_{ \mathbb{A}_r}:   \mathbb{A}_r\rightarrow \mathbb{D} $ is proper.} 

  \vspace{6 pt} 

Observe that when  $|z|=1$, we have $|f_\infty(z)|=1$. By the symmetry $f_\infty(1/\overline{z})=1/ \overline{f_\infty(z)}$ and the identity $f_\infty(r^2z)=f_\infty(z)$, we have 
$$1/\overline{f_\infty(z)}= f_\infty(r^2/\overline{z}).$$
This implies that when $|z|=r$, we have $|f_\infty(z)|=1$. 

Since $f_\infty|_{ \mathbb{A}_r}: \mathbb{A}_r\rightarrow \mathbb{C}$ is a  non-constant  holomorphic function,  by the maximum modulus principle, we have
$$|f_\infty(z)|<\max_{\zeta\in \partial\mathbb{A}_r} |f_\infty(\zeta)|=1, \ \forall z\in \mathbb{A}_r.$$

These properties imply that $f_\infty(\mathbb{A}_r)=\mathbb{D}$ and $f_\infty|_{ \mathbb{A}_r}: \mathbb{A}_r\rightarrow \mathbb{D}$ is a proper map. The degree of $f_\infty|_{ \mathbb{A}_r}$, which can be seen from the fact $(f_\infty|_{ \mathbb{A}_r})^{-1}(0)=\{p_1, \cdots, p_e\}$, is exactly $e$. 

  \vspace{6 pt} 

\textbf{Step 4. }   {\it  The boundary degree   $\deg(f_\infty|_{C_r})=\delta$.} 

  \vspace{6 pt} 

 The degree can be obtained by the argument principle 
\bess 
\deg(f_\infty|_{\partial \mathbb{D}})&=&\frac{1}{2\pi} \int_{\partial \mathbb{D}} d \arg f_\infty(z)=\frac{1}{2\pi i} \int_{\partial \mathbb{D}} \frac{f_\infty'(z)}{f_\infty(z)}dz \\
&=& \lim_{N\rightarrow \infty} \frac{1}{2\pi i} \int_{\partial \mathbb{D}} \frac{f_N'(z)}{f_N(z)}dz\ \ \ (\text{by uniform convergence})   \\
&=&\lim_{N\rightarrow \infty} [(eN+e)-(eN+\delta)] \text{ (by argument principle)}\\ 
&=&e-\delta.
\eess 
Therefore $\deg(f_\infty|_{C_r})=e-\deg(f_\infty|_{\partial \mathbb{D}})=\delta.$
\end{proof}  

\begin{rmk} A  byproduct of the proof of Theorem \ref{model-map}  is the following fact:
Given $e$ points  $p_1, \cdots, p_e\in \mathbb{A}_r$, define $\phi: \mathbb{C}^*\rightarrow \mathbb{\widehat{C}}$  by 
$$\phi(z)= {z^{-\delta}} B_{0}(z)\prod_{j\geq 1} \Big(B_{j}(z)B_{-j}(z)\Big),$$
then the restriction $\phi|_{\mathbb{A}_r}$ is  a proper holomorphic map from $\mathbb{A}_r$ onto $\mathbb{D}$  if and only if $|p_1\cdots p_e|=r^\delta$. 
\end{rmk}

\section{Model space} \label{model}

Let $S$ be a topological space,  $S^{(e)}$ be the {\it $e$-fold symmetric product space} of $S$,  consisting of all   unordered  $e$-tuples  $\{z_1, \cdots, z_e\}$ on $S$,    not necessarily distinct.  The space $S^{(e)}$ is endowed the quotient topology with respect to the projection 
$$p_S: S^e \rightarrow S^{(e)}, \ (z_1, \cdots, z_e)\mapsto \{z_1, \cdots, z_e\}.$$

Given a number $r\in (0,1)$, two  integers $e>\delta>0$,  and  $\{p_1, \cdots, p_e\}\in \mathbb{A}_r^{(e)}$ with $|p_1\cdots p_e|=r^{\delta}$, it is known from Theorem \ref{proper-a-d} that there is a unique proper holomorphic map $f: \mathbb{A}_r \rightarrow \mathbb{D}$ of degree 
$e$ with 
$$f^{-1}(0)=\{p_1, \cdots, p_e\}, \ f(1)=1 \text{\ and }  {\deg}(f|_{C_r})=\delta.$$
Therefore, there is a bijection between the set $\boldsymbol{A}(e,\delta)$ and
 $$Z=\{(r, \{p_1, \cdots, p_e\}); r\in (0,1),  \{p_1, \cdots, p_e\}\in \mathbb{A}_r^{(e)},  |p_1\cdots p_e|=r^{\delta}\}.$$
  Clearly, $Z$ is a topological subspace of $(0,1)\times  (\mathbb{C}^*)^{(e)}$.

\begin{lem}\label{top} The following bijection is a homeomorphism
$$\omega:  
\begin{cases}\boldsymbol{A}(e,\delta) \rightarrow Z , \\
f\mapsto (r_f, f^{-1}(0)).
\end{cases}$$
\end{lem}
 For this, we will not distinguish the spaces $\boldsymbol{A}(e,\delta)$, $Z$ in the following discussion.
 The proof of Lemma \ref{top} is left to the readers.


%
%

%
%

The aim of this section is to study the topology of the   space $\boldsymbol{A}(e,\delta)$.  Before that, we first look at   an example, in order to have a picture in mind.

\subsection{Example: dynamically meaningful M\"obius band and  toroid}

Fix  $r\in(0,1)$,  let's consider the space  $\boldsymbol{A}_r(2,1)$  of all proper holomorphic maps $f: \mathbb{A}_r \rightarrow \mathbb{D}$ of degree two with $f(1)=1$.
In this case,    
$$\boldsymbol{A}_r(2,1)=\{\{p_1,p_2\}\in \mathbb{A}_r^{(2)};  |p_1 p_2|=r\}.$$

In the following, we shall give a description of   $\boldsymbol{A}_r(2,1)$. We will see that $\boldsymbol{A}_r(2,1)$ is actually a  3-D solid torus or a {\it toroid},  containing a  M\"obius band as a dynamically  meaningful subspace. 

Set $p_1= r^{\rho_1}e^{2\pi i t_1},  p_2= r^{\rho_2}e^{2\pi i t_2}$, where $\rho_1, \rho_2\in (0,1)$, $t_1, t_2\in S^1=\mathbb{R}/\mathbb{Z}$.

By changing coordinates,  the space $\boldsymbol{A}_r(2,1)$, viewed as  a set, can be identified as the union of the following two sets
$$ \boldsymbol{A}^{L}_r(2,1)=\{(\rho_1, \rho_2, t_1, t_2)\in (0,1)^2\times \mathbb{T}^2; \rho_1>\rho_2, \rho_1+\rho_2=1\},$$
$$\boldsymbol{A}^{B}_r(2,1)=\{({1} /{2}, {1}/{2}, t_1, t_2);  (t_1, t_2)\in  \mathbb{T}^2\}  \text { modulo order in } t_1, t_2. $$
Each element in $\boldsymbol{A}^{L}_r(2,1)$ is determined  by a triple $(\rho_1,  t_1, t_2)\in (1/2,1)\times \mathbb{T}^2$.  Therefore $\boldsymbol{A}^{L}_r(2,1)$ is homeomorphic to $ (1/2,1)\times \mathbb{T}^2$.    
Note that each map $f$ in  $\boldsymbol{A}^{L}_r(2,1)$ is {\it leaned} in the sense that the two pre-images of $0$ have different  moduli.

The set $\boldsymbol{A}^{B}_r(2,1)$  consists of {\it balanced} maps $f$, in the sense that  the two pre-images of $0$ have same  moduli.  As a topological space,  $\boldsymbol{A}^{B}_r(2,1)$ is homeomorphic to ${S^1}^{(2)}$.
  To visualize ${S^1}^{(2)}$, we identify each point of $S^1$ with $e^{2\pi i t}$. Define the symmetric function ${\rm sym}_2: {S^1}^{(2)}\rightarrow  \mathbb{C}^2$ by
$${\rm sym}_2(\{e^{2\pi i t_1}, e^{2\pi i t_2}\})=(e^{2\pi i t_1}+e^{2\pi i t_2}, e^{2\pi i t_1} \cdot e^{2\pi i t_2}).$$
Clearly ${\rm sym}_2$ is injective.
Let $u=t_1+t_2, v=t_1-t_2$, then 
$${\rm sym}_2(\{e^{2\pi i t_1}, e^{2\pi i t_2}\})=e^{\pi i u}(2\cos(\pi v), e^{\pi i u}).$$
Therefore the image ${\rm sym}_2 ( {S^1}^{(2)})$ can be parameterized by the  parameters $(u,v)$, and it is homeomorphic to the image  of $\gamma:  [0,1]\times [-1,1]\rightarrow \mathbb{R}^3$
$$\gamma(u, v)  
 =\left(
\begin{array} {c}
(2+\cos(\pi v)\cos(\pi u))\cos (2\pi u )\\
(2+\cos(\pi v)\cos(\pi u))\sin (2\pi u )\\
\cos(\pi v)\sin(\pi u) 
\end{array}
\right).$$
 The image of $\gamma$   is  foliated by $\cup_{u\in [0,1]} \gamma(u,[-1,1])$.
To understand this graph,  consider a line segment $\ell=[-1,1]$  in $\mathbb{R}^3$ moving along the  round circle of radius $2$ centered at the origin  in the $xy$-plane.  At each point 
$z_u=(2\cos (2\pi u), 2\sin(2\pi u), 0)$,  the line segment $\ell$ is  perpendicular to the circle and with plane angle  $\pi u$, with the midpoint of $\ell$  exactly $z_u$.   One may find that the image of $\gamma$ is exactly a
M\"obius band.

\begin{figure}[htpb]
  \setlength{\unitlength}{1mm}
  \centering{
 \includegraphics[width=80mm]{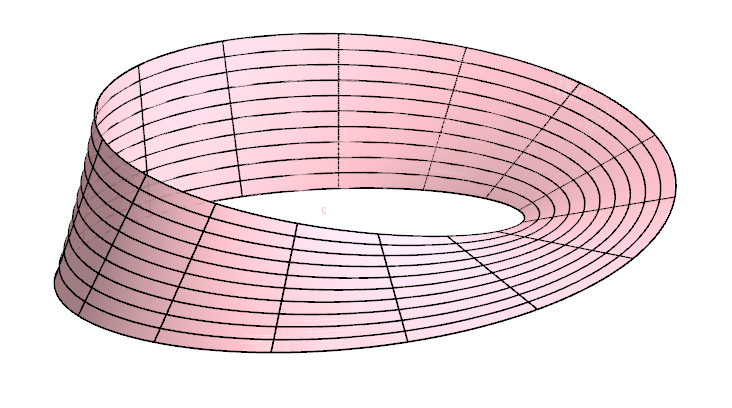} }
 \caption{The space $\boldsymbol{A}^{B}_r(2,1)$ of {\it balanced} proper holomorphic maps $f: \mathbb{A}_r \rightarrow \mathbb{D}$ of degree two with $f(1)=1$,
  is naturally homeomorphic to a M\"obius band.}
 \end{figure}

Finally, to visualize    $\boldsymbol{A}_r(2,1)$, we glue the  inner boundary $\mathcal{B}_{in}=\{({1}/{2}, {1}/{2}, \\ t_1, t_2);  t_1, t_2 \in S^1\}$
 of $\boldsymbol{A}^{L}_r(2,1)$ 
so that the  two points $({1}/{2}, {1}/{2}, t_1, t_2)$ and $({1}/{2}, {1}/{2}, t_2, t_1)$  collapse to one point. In this way,  $\mathcal{B}_{in}$ collapses to  a M\"obius band, as shown above. It then turns out that $\mathcal{B}_{in}\cup\boldsymbol{A}_r^L(2,1)$     collapses to a toroid, giving the  topology of $\boldsymbol{A}_r(2,1)$.  

The rigorous proof of these intuitive descriptions is the task of next part.

\subsection{Global topology of model space} \label{global-top}

The following result says that  global topology of   $\boldsymbol{A}(e,\delta)$  is very standard.

\begin{thm}\label{model-space} The model space   $\boldsymbol{A}(e,\delta)$ is homeomorphic to 
$S^1\times \mathbb{R}^{2e-1}.$\end{thm} 
\begin{proof} The idea of the proof is  inspired by one way of finding a leak in a tire: 
first, inflate the tube,  then use the hissing noise to locate the hole.
Applying in our case, we first `inflate' the annulus $\mathbb{A}_r$ to the punctured plane $\mathbb{C}^*$, then
using   the symmetric function to detect the topology of symmetric product space, and their subspaces.

As a warm-up,
let's first consider the subspace $\mathcal{L}$ of $(\mathbb{C}^*)^{(e)}$:
$$\mathcal{L}=\{\{\zeta_1, \cdots, \zeta_e\}\in  (\mathbb{C}^*)^{(e)}; |\zeta_1\cdots \zeta_e|=1\}.$$
Observe that $\mathcal{L}$ can be identified to  $\mathbb{C}^{e-1}\times S^1$ by the symmetric function: 
$${\rm sym}_e:  \{\zeta_1, \cdots, \zeta_e\}  \mapsto  (c_1, \cdots, c_{e-1}, c_e)\in \mathbb{C}^{e-1}\times S^1, $$
where $c_k$ are defined in the following way
$$(z-\zeta_1)\cdots(z-\zeta_e)=z^e+\sum_{k=1}^e (-1)^k c_k z^{e-k}.$$

In the following,  we shall  show that fix any $r\in (0,1)$, the space 
$$\boldsymbol{A}_r(e,\delta)=\{\{p_1, \cdots, p_e\}\in \mathbb{A}_r^{(e)};  |p_1\cdots p_e|=r^{\delta}\}$$
 is homeomorphic to 
$S^1\times \mathbb{R}^{2e-2}.$
To this end, it suffices to prove that $\boldsymbol{A}_r(e,\delta)$ is homeomorphic to $\mathcal{L}$. Let $\phi_r: \mathbb{A}_r\rightarrow \mathbb{C}^*$ be the homeomorphism defined by
$$\phi_r(p)=\frac{|p|-r}{1-|p|}\cdot\frac{p}{|p|}.$$ 
Then we  define a map $\Phi_r: \boldsymbol{A}_r(e,\delta)\rightarrow \mathcal{L}$ by
$$\Phi_r(\{p_1, \cdots, p_e\})=\Bigg\{ \frac{\phi_r(p_1)} {|\phi_r(p_1)\cdots\phi_r(p_e)|^{1/e}}, \cdots,    \frac{\phi_r(p_e)} {|\phi_r(p_1)\cdots\phi_r(p_e)|^{1/e}}\Bigg\}.$$
It's clear that $\Phi_r$ is continuous.  To see that $\Phi_r$ is a homeomorphism,  we need construct  an inverse of $\Phi_r$.
To  do this,  for any $\{\zeta_1, \cdots, \zeta_e\}\in \mathcal{L}$,  consider  the function $g: (0, +\infty)\rightarrow \mathbb{R}$ defined by 
$$g(t)=\prod_{k=1}^e\frac{|\zeta_k|+t r}{|\zeta_k|+t}.$$
Observe that $g$ is monotonically decreasing and satisfies
 $$\lim_{t\rightarrow 0^+}g(t)=1,  \lim_{t\rightarrow +\infty}g(t)=r^e.$$
So there is a  unique positive number $t_0=t_0(\zeta_1,\cdots, \zeta_e)$  satisfying $g(t_0)=r^{\delta}$.
Consider the map $\Psi_r: \mathcal{L} \rightarrow \boldsymbol{A}_r(e,\delta)$  defined by
$$\Psi_r(\{\zeta_1, \cdots, \zeta_e\})=\Bigg\{ \frac{|\zeta_1|+t_0 r}{|\zeta_1|+t_0}\cdot\frac{\zeta_1}{|\zeta_1|}, \cdots,     \frac{|\zeta_e|+t_0 r}{|\zeta_e|+t_0}\cdot\frac{\zeta_e}{|\zeta_e|}\Bigg\}.$$
One may verify that $\Psi_r\circ \Phi_r=id_{\boldsymbol{A}_r(e,\delta)}$. This means that $\Phi_r$ is both injective and surjective, therefore a homeomorphism from  $\boldsymbol{A}_r(e,\delta)$ onto $\mathcal{L}$.

Finally, define the map  $H:\boldsymbol{A}(e,\delta)\rightarrow \mathbb{R}\times \mathbb{C}^{e-1}\times S^1\simeq \mathbb{R}^{2e-1}\times S^1$ by
$$H(r, \{p_1, \cdots, p_e\})=\bigg(\tan\Big(\big(r-\frac{1}{2}\big)\pi \Big),\  {\rm sym}_e \circ  \Phi_r (\{p_1, \cdots, p_e\})\bigg).$$
It is easy to see that $H$ is a homeomorphism.
\end{proof}

Now, for the model space   $\mathbf{M}_\sigma$ introduced in Section \ref{model-map0}, we have: 

\begin{cor} \label{model-space1} The model space  $\mathbf{M}_\sigma$  is homeomorphic to $\mathbb{R}^{4d-4-n}\times \mathbb{T}^n$.
\end{cor}
\begin{proof}
It's known from Lemma \ref{model-space0} and Theorem \ref{model-space}  that the model space $\mathbf{M}_\sigma$ is   homeomorphic to (note  that $d=d_1+\cdots+d_{n+1}$)
$$\mathbb{R}^{2d_1-2}\times \mathbb{R}^{2d_{n+1}-2}\times \prod_{k=1}^{n}\Big(\mathbb{R}^{2(d_{k}+d_{k+1})-1}\times S^1\Big) \simeq  \mathbb{R}^{4d-4-n}\times \mathbb{T}^n.$$  \end{proof}

 \section{The covering property of $\rho$}\label{cov}

 In this  section,  we shall prove the covering  property of  the map  $\rho: \mathcal{F}^{\chi_0}_{\sigma,\boldsymbol{d}}\rightarrow \mathbf{M}_\sigma$  (introduced in Section \ref{dynamics}) defined  by 
 $$\rho((f, \nu))= \boldsymbol{m}(f)=(m_{0}, m_{{\infty}}, m_{1}, \cdots, m_{{n}}). $$

\begin{thm} \label{covering-p}  The map  $\rho:  \mathcal{F}^{\chi_0}_{\sigma,\boldsymbol{d}}\rightarrow  \mathbf{M}_\sigma$ is a 
  covering map.   
\end{thm}

 Recall that, a map $p: X\rightarrow Y$ between two topological spaces $X$ and $Y$ is a {\it covering map}  if  for every point $y\in Y$, there is a neighborhood $V$ of $y$ such that
 every component  of $p^{-1}(V)$ maps homeomorphically onto $V$.
 
 Note that in Theorem \ref{covering-p}, we don't assume the connectivity of  $\mathcal{F}^{\chi_0}_{\sigma,\boldsymbol{d}}$.
  The proof bases on the  quasi-conformal surgery  and the Thurston-type theory, developed by Cui and Tan \cite{CT}. 

\subsection{C-equivalence}   The following definitions are borrowed from \cite{CT}, with slightly different but essentially equivalent statements. 

 Let $g: \widehat{\mathbb{C}}\rightarrow \widehat{\mathbb{C}}$ be a branched cover with degree at least two.
  Let  $C_g$ be its critical set,  and 
 $P_g=\overline{\bigcup_{k>0} g^k(C_g)}$
 its   post-critical set,  $P'_g$ the accumulation set of $P_g$.

We say that  $g$ is  {\it semi-rational}  if  $P'_g$ is finite (or empty); and in case $P'_g\neq \emptyset$, the map $g$ is
holomorphic in a neighborhood of $P'_g$ and every periodic point in $P'_g$ is either
attracting or super-attracting.

Two semi-rational maps $g_1$ and $g_2$ are called {\it c-equivalent}, if there exist a pair
$(\phi, \psi)$ of homeomorphisms of $\widehat{\mathbb{C}}$ and a neighborhood $U_0$ ($=\emptyset$ when  $P'_g=\emptyset$) of $P'_{g_1}$
such that:

(a).  $\phi\circ g_1=g_2\circ \psi$;

(b).  $\phi$ is holomorphic in $U_0$;

(c).  the two maps $\phi$ and $\psi$ satisfy $\phi|_{P_{g_1}\cup U_0}=\psi|_{P_{g_1}\cup U_0}$;

(d).  the two maps $\phi$ and $\psi$  are isotopic  rel $P_{g_1}\cup U_0$.

In this case,  we say  that  $g_1$ and $g_2$ are  c-equivalent via $(\phi, \psi)$. 

\subsection{Proof of Theorem \ref{covering-p}}
The proof is built on two propositions.

\begin{pro}\label{surj} The map $\rho:\mathcal{F}^{\chi_0}_{\sigma,\boldsymbol{d}}\rightarrow  \mathbf{M}_\sigma$ is surjective.
\end{pro}
 
 \begin{proof}  It is equivalent to show that for any model map  $\boldsymbol{m}\in \mathbf{M}_\sigma$,  the fibre $\rho^{-1}(\boldsymbol{m})$ is non-empty. 
 The proof consists of three steps: first construct a branched cover with the prescribed  holomorphic model $\boldsymbol{m}$, then apply (a special case of) Cui-Tan's Theorem to generate a rational map,
 finally show that  this rational map realizes the original model $\boldsymbol{m}$.

 \vspace{5pt}
 
 \textbf{Step 1.} \textit{Constructing a branched cover with prescribed  model map.}

 \vspace{5pt}
 
 Write $\boldsymbol{m}=(m_0, m_\infty,   m_1,\cdots, m_n)$, denote the domain of definition  of $m_k$ by $D(m_k)$, $k\in \mathcal{I}=\{0,\infty, 1,\cdots, n\}$.
 Choose a sequence of numbers
 $$1<r_1<R_1<r_2<R_2<\cdots<r_n<R_n<R<+\infty$$
 satisfying that $\mod(D(m_k))=\frac{1}{2\pi}\log(R_k/r_k)$ for $1\leq k\leq n$.  Let $B_0=\mathbb{D}$, $B_\infty=\{z\in \mathbb{\widehat{C}}; |z|>R\}$ and $B_k=\{r_k<|z|<R_k\}$
 for $1\leq k\leq n$.  For each $k\in \mathcal{I}$, there is a conformal embedding $e_k: D(m_k)\hookrightarrow \widehat{\mathbb{C}}$, whose image is exactly $B_k$. We assume that $e_0=id|_{\mathbb{D}}$, $e_\infty(0)=\infty$,
 and for each $k\in\{1,\cdots, n\}$, the point $e_k(1)$ is on the outer boundary of $B_k$ if $\epsilon_k=+$; on the inner boundary of $B_k$ if $\epsilon_k=-$ (recall that $\chi_0=(\epsilon_1, \cdots, \epsilon_n)$ is a symbol vector, see Section \ref{dynamics}). We construct a branched cover  of $ \widehat{\mathbb{C}}$ as follows:
  $$g= \begin{cases} 
e_{\tau(k)}\circ m_k\circ e_k^{-1}, &\text{ on  } B_k,  \ k\in\mathcal{I}, \\
\text{quasi-regular interpolation},  &\text{ on } \mathbb{\widehat{C}}\setminus \cup_{k\in \mathcal{I}} B_k.
\end{cases}$$

Note that the boundary degrees satisfy the inequality (see Section \ref{dynamics})
 $$\sum_{k=1}^{n+1}\frac{1}{d_k}<1.$$
As is interpreted in \cite{CT}, this inequality  is equivalent to the absence of Thurston obstruction.
 Therefore by a special case of Cui-Tan's Theorem  \cite[Section 6.2 and Lemma 6.2]{CT}, the map $g$ is c-equivalent 
to a rational map $h$, via a pair of homeomorphisms, say $(\phi_0,\phi_1)$.   

 \vspace{5pt}

\textbf{Step 2. }{ \it The Julia set $J(h)$ is a Cantor set of circles.} 

 \vspace{5pt}

 By the definition of c-equivalence, the maps $\phi_0, \phi_1$ are holomorphic and identical in a  
neighborhood $U$ of $P'_g$.  Note that  $P'_g=\{0\}$ in the Type I case,  and $P'_g=\{0,\infty\}$ in the Type II, III cases.
We  further assume that $\phi_0$ and $\phi_1$ both fix $0$ and $\infty$.  
By a lifting process,  we can get a sequence of homeomorphisms $\phi_k$ satisfying  that   $\phi_k\circ g=h \circ \phi_{k+1}$ and $\phi_k, \phi_{k+1}$ are isotopic rel $g^{-k}(U)\cup P_g$.

Let $U^k_0$ be the component of $g^{-k}(U)$ containing $0$,  and $U^k_\infty$  the component of $g^{-k}(U)$ containing $\infty$.  By the suitable choice of $U$, we assume 
$U^k_0\Subset U^{k+1}_0, U^k_\infty\Subset U^{k+1}_\infty$ for all $k\geq 0$. Choose a large integer $\ell>0$ so that $P_g\subset U^\ell_0\bigcup U^\ell_\infty$. This  implies that $P_h\subset \phi_\ell(U^\ell_0\cup U^\ell_\infty)$.  Set $A=\mathbb{\widehat{C}}\setminus  \phi_\ell(\overline{U^\ell_0\cup U^\ell_\infty})$, then $B=h^{-1}(A)\Subset A$ and each component of  $B$ is an annulus.
 It follows that
$J(h)=\bigcap_k{h^{-k}(A)}$ and it  is a Cantor set of circles.

 \vspace{5pt}

\textbf{Step 3. }{\it  There  is $h_0\in \langle h\rangle$  and  a boundary marking $\nu_0$ of $h_0$, so that $(h_0, \nu_0)\in \mathcal{F}^{\chi_0}_{\sigma,\boldsymbol{d}}$ and $\rho((h_0, \nu_0))=\boldsymbol{m}$.}

 \vspace{5pt}

By  suitable choice of representative in the isotopy class,  we may assume $\phi_0$ maps $D_0\cup D_\infty$ homeomorphically onto  $D_0(h)\cup D_\infty(h)$, where $D_w(h)$ is the Fatou component of $h$ containing $w\in\{0,\infty\}$.
We assume further that the maps $\phi_k$ constructed in Step 2 are quasi-regular.   Their dilatations are not uniformly bounded, however,  the  dilatations of
$\phi_k|_{B_j}$ are uniformly bounded  for any $j\in \mathcal{I}$.
Since $\bigcup_k U^k_0=B_0$ and $\bigcup_k U^k_\infty=B_\infty$, we have that 
$\phi_k|_{B_j}$ converges uniformly to a conformal isomorphism, say  $\alpha_j: B_j\rightarrow D_j(h)$, where $D_j(h)$ is the corresponding  Fatou component of $h$. 
These $\alpha_j$'s satisfy that $h|_{D_j(h)}\circ \alpha_j=\alpha_{\tau(j)}\circ g|_{B_j}$.  The marking for $g$ induces a marking $\nu$ of $h$. 
Let $\phi$ be the M\"obius transformation mapping the triple $(0,\alpha_0(1), \infty)$ to $(0,1, \infty)$.
Then the marked  map $(h_0, \nu_0)=(\phi\circ h \circ \phi^{-1}, \phi\circ \nu)$ satisfies  $(h_0, \nu_0)\in \mathcal{F}^{\chi_0}_{\sigma,\boldsymbol{d}}$ and $\rho((h_0, \nu_0))=\boldsymbol{m}$.

The surjectivity  of $\rho$ then follows. 
 \end{proof}

   %

\begin{pro}    For every  model map $\boldsymbol{m}\in \mathbf{M}_\sigma$, there is a neighborhood $\mathbf{N}$ of  $\boldsymbol{m}$ satisfying that for each marked map $(f, \nu)\in \rho^{-1}(\boldsymbol{m})$, 
there is a  neighborhood $\mathbf{U}$ of $(f,\nu)$ so that $\rho|_{\mathbf{U}}: \mathbf{U}\rightarrow \mathbf{N}$ is a homeomorphism.  
\end{pro}
\begin{proof}  
 For any  $\boldsymbol{m}=(m_0,m_\infty, m_1, \cdots, m_n)\in \mathbf{M}_\sigma$, 
%
    take $(f, \nu)\in \rho^{-1}(\boldsymbol{m})$.
 Suppose that $(f|_{D_0}, f|_{D_\infty}, f|_{D_1}, \cdots, f|_{D_n})$ is conformally conjugate  to  $\boldsymbol{m}$ by the conformal isomorphism
$\kappa_{f}=(\kappa_{0}, \kappa_{\infty}, \kappa_{1},\cdots, \kappa_{{n}})$. 

We  then choose a number $r\in (0,1)$,   sufficiently close to $1$  and satisfying the following two properties: 

\vspace{5 pt}

(P1)  For each $k\in\{0,\infty\}$, the disk   $\mathbb{D}_{r}$ contains all critical values of  ${m}_k$.
   
 (P2)   For  each $k\in\{1,\cdots, n\}$,  the set  $m_{\tau(k)}^{-1}(\mathbb{D}_{r})$ contains all the critical values of the model map  ${m}_k: \mathbb{A}_{r_k}\rightarrow \mathbb{D}$.
  
  \vspace{5 pt}

Take another number $R\in (r, 1)$, there is a small polydisk-type  neighborhood $\mathbf{N}=\prod_{j\in\mathcal{I}}N_j$ of $\boldsymbol{m}$,
  such that  for all $\tilde{\boldsymbol{m}}=(\tilde{m}_0,\tilde{m}_\infty, \tilde{m}_1, \cdots, \tilde{m}_n)\in \mathbf{N}$, 
  the properties (P1)(P2) still hold (one should replace $m_*$ by $\tilde{m}_*$ in the statement), and 
  $$\tilde{m}_j^{-1}(\mathbb{D}_r) \Subset {m}_j^{-1}(\mathbb{D}_R),  j\in\{0,\infty\};   \ A^r_{k, \tilde{\boldsymbol{m}}} \Subset A^R_{k, \boldsymbol{m}}, \ k\in\{1,\cdots, n\},$$
where   $A^a_{k, \tilde{\boldsymbol{m}}}=(\tilde{m}_{\tau(k)}\circ \tilde{m}_k)^{-1}(\mathbb{D}_{a}),  a\in\{r, R\}$.
  
  


We then construct a quasi-regular map as follows  
$$g_{\tilde{\boldsymbol{m}}}= \begin{cases}  \kappa_{\tau(j)}^{-1}\circ \tilde{m}_j\circ \kappa_j, &\text{ in } \kappa_j^{-1} ( \tilde{m}_j^{-1}(\mathbb{D}_r)), j\in\{0,\infty\}, \\
\kappa_{\tau(j)}^{-1}\circ \tilde{m}_j\circ \kappa_j, &\text{ in } \kappa_j^{-1}(A^r_{j,\tilde{\boldsymbol{m}}}),  j\in\{1,\cdots, n\}, \\
f,   &\text{ in } \widehat{\mathbb{C}}\setminus  U_R,  \\
\text{quasi-regular interpolation},  &\text{ in  the rest,}
\end{cases}$$
where 
$$U_R=\Big(\bigcup_{j=0, \infty}\kappa_j^{-1}({m}_j^{-1}(\mathbb{D}_R))\Big) \bigcup  \Big( \bigcup_{1\leq j \leq n}\kappa_j^{-1}(A^R_{j, \boldsymbol{m}}) \Big).$$

 By careful gluing and suitable choices of interpolations, it is reasonable to require that   $g_{\tilde{\boldsymbol{m}}}$ moves continuously with respect to $\tilde{\boldsymbol{m}}\in \mathbf{N}$   and $g_{\boldsymbol{m}}=f$.
Then we  pull back  the standard complex structure   defined in a   neighborhood of the attracting cycles of $g_{\tilde{\boldsymbol{m}}}$ by successive iterates, and get a $g_{\tilde{\boldsymbol{m}}}$-invariant complex structure, whose Beltrami differential is denoted by $\mu_{\tilde{\boldsymbol{m}}}$.

 Let $\phi_{\tilde{\boldsymbol{m}}}$ be a  quasiconformal map fixing $0,1,\infty$ and solving $\overline{\partial} \phi_{\tilde{\boldsymbol{m}}}=\mu_{\tilde{\boldsymbol{m}}} {\partial} \phi_{\tilde{\boldsymbol{m}}}$. Then   $f_{\tilde{\boldsymbol{m}}}=\phi_{\tilde{\boldsymbol{m}}}\circ g_{\tilde{\boldsymbol{m}}} \circ \phi_{\tilde{\boldsymbol{m}}}^{-1}$ is a rational map with 
 $f_{\tilde{\boldsymbol{m}}}(1)=1$.   The boundary marking $\nu$ of $f$ induces a boundary 
 marking $\nu_{\tilde{\boldsymbol{m}}}=\phi_{\tilde{\boldsymbol{m}}}\circ \nu$ for $f_{\tilde{\boldsymbol{m}}}$.

To finish, we need prove   $\rho((f_{\tilde{\boldsymbol{m}}}, \nu_{\tilde{\boldsymbol{m}}}))=\tilde{\boldsymbol{m}}$ . 
This is a consequence of the following fact, whose proof is  similar to  \cite[Lemma 5.10]{M}  (compare also the Step 3 in the proof of Proposition \ref{surj}).  For this, we omit the details.   

\vspace{5 pt}

\textbf{Fact }  The conformal conjugacy class of the  homomorphic model $\tilde{\boldsymbol{m}}=(\tilde{m}_0,\tilde{m}_\infty, \tilde{m}_1, \cdots, \tilde{m}_n)$ is
 uniquely determined by the conformal conjugacy class of its restrictions $(\tilde{m}_0|_{\tilde{m}_0^{-1}(\mathbb{D}_r)},\tilde{m}_\infty|_{\tilde{m}_\infty^{-1}(\mathbb{D}_r)}, \tilde{m}_1|_{A^r_{1,\tilde{\boldsymbol{m}}}}, \cdots, \tilde{m}_n|_{A^r_{n,\tilde{\boldsymbol{m}}}})$.
\end{proof}

\vspace{5 pt}

 

 \section{The finiteness property of $\rho$}  \label{finite}
 We will go one step further in this section.
 By Theorem \ref{covering-p},  we know that  $\rho: \mathcal{F}^{\chi_0}_{\sigma,\boldsymbol{d}}\rightarrow \mathbf{M}_\sigma$ is a covering map. The {\it mapping degree} of $\rho$, denoted by $\deg(\rho)$, is defined as the 
 cardinality  of the fibre $\rho^{-1}(\boldsymbol{m})$, where $\boldsymbol{m}$  can be any model map in $\mathbf{M}_\sigma$.
 In this section, we will show

\begin{thm} \label{covering}    The mapping degree of $\rho: \mathcal{F}^{\chi_0}_{\sigma,\boldsymbol{d}}\rightarrow \mathbf{M}_\sigma$ is given by 
  $${\rm deg}(\rho)=\bigg(1-\sum_{k=1}^{n+1}\frac{1}{d_k}\bigg){\rm lcm}(d_1, \cdots, d_{n+1}) .$$
  In particular,     $\rho$ is finite-to-one.\footnote{The notation `lcm'  means the least common multiple.}
\end{thm}

 This finiteness property of $\rho$  will be essential when we study the global topology of the marked hyperbolic component in the next section (see the proof of Theorem \ref{marked-comp}).  
 To explain why it is essential, let's consider an example of covering map $\zeta: X\rightarrow S^1$ in dimension one. It's  known from algebraic topology that if $\zeta$ is finite-to-one, then $X$ is homeomorphic to $S^1$; if $\zeta$ is infinite-to-one, then  $X$ is homeomorphic to $\mathbb{R}$. Therefore the topology of $X$ is  related to the mapping degree of $\zeta$.  The same reason works for our (higher dimensional) case.

 The idea of the proof of Theorem \ref{covering} is due to Guizhen Cui, using twist deformation techniques  \cite{C}  and  the Thurston-type theorem for hyperbolic rational maps, developed by Cui and Tan \cite{CT}.   
  
  \subsection{The twist map. }  Recall that $\mathbb{A}_r=\{r<|z|<1\}$. The standard  {\it twist function}
    $t_{r}: \mathbb{A}_r\rightarrow \mathbb{A}_r$ is  defined by
    $$t_{r}(z)= z e^{2\pi i\frac{|z|-r}{1-r}}, \  z\in \mathbb{A}_r.$$
    It's clear that $t_{r}$ is a homeomorphism and $t_{r}|_{\partial \mathbb{A}_r}=id$.
  
  
  Let $A\subset \mathbb{\widehat{C}}$ be an annulus, whose boundaries are Jordan curves. 
   We define the {\it twist map along} $A$ by 
  $$T_A(z)= \begin{cases} z, & z\in  \widehat{\mathbb{C}}\setminus A, \\
 \phi^{-1}\circ t_{r}\circ \phi (z), & z\in   A,
\end{cases}$$
where    $\phi: A\rightarrow \mathbb{A}_r$ (here $r=e^{-2\pi {\rm mod}(A)}$) is a  conformal isomorphism.
Note that $T_A$ does not depend on the choice of $\phi$.

 \subsection{Proof of  Theorem \ref{covering}}   
 Let $\boldsymbol{m}=(m_0, m_\infty, m_1, \cdots, m_n)\in \mathbf{M}_\sigma$.  To evaluate the cardinality of $\rho^{-1}(\boldsymbol{m})$, throughout this section, we require $\boldsymbol{m}$ to be `generic' in the following sense:

\vspace{5 pt}
 
  (C1).  $R_{\theta_1}\circ m_0=m_0\circ  R_{\theta_2}  \Longrightarrow   R_{\theta_1}= R_{\theta_2}=id;$
  
  (C2). $R_{\theta_1}\circ m_0\neq m_\infty\circ  R_{\theta_2}, \ \forall   \  \theta_1, \theta_2\in [0, 2\pi);$
  
  (C3). $m_j\circ R_{\theta}=m_j  \Longrightarrow   R_{\theta}=id, \   \forall j=1, \cdots, n,$
  
  where $R_\alpha(z)=e^{i\alpha}z$.
 
\vspace{5 pt}
In fact, these technical assumptions  exclude the rotation symmetries of  $\boldsymbol{m}$, and they will be used in the proofs of Lemma \ref{non-c-e} and Theorem \ref{covering}.

   Take  a marked   map  $(f, \nu)\in \rho^{-1}(\boldsymbol{m})$.   For this $f$, let $n=n(f), D_j=D_j(f), A_k=A_k(f)$ be defined as in Section \ref{dynamics}.
 
 Let $T_k$ be the twist map along $A_k$, and $T$ be the twist map along $A:=\widehat{\mathbb{C}}\setminus(\overline{D}_0\cup \overline{D}_\infty)$.
 Note that:  (1). The post-critical set $P_f$ is contained in $D_0\cup D_\infty$;  
 (2). $T$ and $T_k$ are isotopic rel $D_0\cup D_\infty$;  (3). $T\circ T_k$ and $T_k\circ T$ are isotopic  rel $D_0\cup D_\infty$;  (4). $T_k\circ T_j=T_j\circ T_k$.

\begin{lem}  For any $(g, \mu)\in \rho^{-1}(\boldsymbol{m})$, the  map $g$ is c-equivalent to $T^{k}\circ f \circ T_1^{k_1}\circ\cdots\circ T_{n+1}^{k_{n+1}}$ for some $(k, k_1, \cdots, k_{n+1})\in \mathbb{Z}^{n+2}$. 
 \end{lem}

 \begin{proof} 
With the similar notation and ordering as $f$,  we denoted the critical Fatou components of $g$  by  $D_j (g), { j\in\mathcal{I}}$. 
 First observe that, there is a homeomorphism $\tau: \widehat{\mathbb{C}}\rightarrow \widehat{\mathbb{C}}$  mapping 
 $\bigcup_{j\in\mathcal{I}}D_j(g)$ holomorphically onto $\bigcup_{j\in\mathcal{I}}D_j$,  fixing $0, \infty$ and sending the marked points of $g$ to that of $f$. 
Then we compare the map  $g_\tau=\tau\circ g\circ \tau^{-1}$  and $f$.  
Note that  $g_\tau|_{A_j}$, $f|_{A_j}$ are both  $d_j$-fold covering maps from $A_j$ to $A=\widehat{\mathbb{C}}\setminus (\overline{D_0}\cup \overline{D_\infty})$, therefore for some suitable choice of integer  $b_j>0$, the restriction $T^{b_j}|_{A}$  can be lifted to a homeomorphism $\zeta_j: A_j\rightarrow A_j$ with $\zeta_j|_{\partial A_j}=id$, as illustrated in the following commutative diagram:
$$
\xymatrix{ & A_j\ar[r]^{g_\tau|_{A_j}}
\ar[d]_{\zeta_j}
&A  \ar[d]^{T^{b_j}|_{A}}\\
& A_j \ar[r]_{f|_{A_j}} &  A}
$$ 
One may observe that $\zeta_j$ is  isotopic to ${T^{a_j}_j}|_{A_j}$ rel $\partial A_j$ for some $a_j\in \mathbb{Z}$.
Now let $B={\rm lcm}(b_1, \cdots, b_{n+1})$ and $B_j=Ba_j/b_j$, then $g_{\tau}$  is c-equivalent to 
$$T^{-B}\circ f \circ T_1^{B_1}\circ\cdots\circ T_{n+1}^{B_{n+1}}$$
via $(id, \psi)$, where $\psi$ is isotopic to $id$ rel $\bigcup_{k\in\mathcal{I}} D_k$.  It follows that $g$ is c-equivalent to 
$T^{-B}\circ f \circ T_1^{B_1}\circ\cdots\circ T_{n+1}^{B_{n+1}}$
via $(\tau, \psi\circ \tau)$.
   \end{proof}
 
 \begin{lem} \label{good-form} For any $(k, k_1, \cdots, k_{n+1})\in \mathbb{Z}^{n+2}$,   the map  $T^k\circ f \circ T_1^{k_1}\circ\cdots\circ T_{n+1}^{k_{n+1}}$ is c-equivalent to $T^{k+\sum_{j=1}^{n+1} k_j}\circ f$.
 \end{lem}
 \begin{proof}  Clearly  $T_1^{k_1}\circ\cdots\circ T_{n+1}^{k_{n+1}}$ conjugate  $T^k\circ f \circ T_1^{k_1}\circ\cdots\circ T_{n+1}^{k_{n+1}}$
 to $T_1^{k_1}\circ\cdots\circ T_{n+1}^{k_{n+1}}\circ T^k\circ f $, which is c-equivalent to $T^{k+\sum_{j=1}^{n+1} k_j}\circ f$ via a pair  $(id, \psi)$ of isotopic homeomorphisms.
  \end{proof}
 
 In the following, set
 $$N={\rm lcm}(d_1, \cdots, d_{n+1}), \  N^*=\Big(1-\sum_{j=1}^{n+1}\frac{1}{d_j}\Big){\rm lcm}(d_1, \cdots, d_{n+1}).$$
 
  \begin{lem}\label{periodic-c-e}   For any integer $\ell\in \mathbb{Z}$, the map $T^\ell\circ f$ is c-equivalent to $T^{N^*+\ell}\circ f$.
 \end{lem}
 
 \begin{proof} It suffices to prove the case $\ell=0$.   Let's consider the action of  the following map
 $$f_N=T^N\circ f \circ T_1^{-N/d_1}\circ\cdots\circ T_{n+1}^{-N/d_{n+1}}$$
 on  the annuli $A_1, \cdots,  A_{n+1}$. Note that $A_j$ is first twisted $-N/d_{j}$ times by  $T_j^{-N/d_{j}}$, then the twist time is multiplied by $d_{j}$ because of the $f$-action, and finally the $T^N$-action
 contributes  $+N$ additional twist times.  Therefore the total twist time of $f_N$ on $A_j$ is $N+d_{j}\times(-N/d_{j})=0$, 
 equal to the twist time of $f$ on $A_j$.
 
 So 
 we can lift  the restriction of the  identity map on $A$ to get a homeomorphism $\zeta_j: A_j\rightarrow A_j$, isotopic to $id|_{A_j}$ rel the boundary $\partial{A_j}$,
 see the following commutative diagram
 $$
\xymatrix{ & A_j\ar[r]^{f|_{A_j}}
\ar[d]_{\zeta_j}
&A  \ar[d]^{id|_{A}}\\
& A_j \ar[r]_{f_N|_{A_j}} &  A}
$$ 

By gluing the maps $\zeta_j, id|_{D_k}$ together, 
we get a homeomorphism $\psi: \widehat{\mathbb{C}}\rightarrow \widehat{\mathbb{C}}$.
Then the map $f$ is c-equivalent,  via the pair  $(id, \psi)$ of isotopic homeomorphisms, to $f_N$, which is c-equivalent to  $T^{N^*}\circ f$ (by Lemma \ref{good-form}).
 \end{proof}

  \begin{lem} \label{non-c-e} For any  $0\leq m_1<m_2<N^*$, the maps  $T^{m_1}\circ f$ and $T^{m_2}\circ f$ are not  c-equivalent.
 \end{lem}
 
 \begin{proof}   If not, suppose that $T^{m_1}\circ f$ and $T^{m_2}\circ f$ are  c-equivalent via $(\phi_0, \phi_1)$.  
 Then by  lifting, there is a sequence of  homeomorphisms $\phi_j: \widehat{\mathbb{C}}\rightarrow \widehat{\mathbb{C}}$
  so that $\phi_j\circ (T^{m_1}\circ f)=(T^{m_2}\circ f)\circ \phi_{j+1}$, and $\phi_j$ is isotopic to $\phi_{j+1}$ rel $P_f$.
 Since $\phi_0$ is holomorphic in a neighborhood of attracting   cycles, it's not hard to see that   the restrictions $\phi_j|_{\bigcup_{k\in\mathcal{I}} D_k}$ converge to a conformal map
 $\phi:  \bigcup_{k\in\mathcal{I}} D_k\rightarrow  \bigcup_{k\in\mathcal{I}} D_k$ as $j\rightarrow \infty$.
 Therefore in the isotopy class of $\phi_0$, there is a pair of homeomorphisms $\psi_0,\psi_1$ so that $T^{m_1}\circ f$ and $T^{m_2}\circ f$ are  c-equivalent via $(\psi_0, \psi_1)$, and that $\psi_0$ and $\psi_1$  are holomorphic in $\bigcup_{k\in\mathcal{I}} D_k$.  The assumptions (C1,C2,C3) imply that 
 $$\psi_0|_{D_k}=\psi_1|_{D_k}=id|_{D_k},   \ \forall \ k \in\mathcal{I}.$$
 
 It turns out that $\psi_0$ is isotopic to $T_1^{k_1}\circ\cdots\circ T_{n+1}^{k_{n+1}}$ rel $\bigcup_{k\in\mathcal{I}} D_k$. 
 Thus  $T^{m_1}\circ f$ and $T^{m_2}\circ f$ are  c-equivalent via $(T_1^{k_1}\circ\cdots\circ T_{n+1}^{k_{n+1}}, \psi)$. In other words, $T^{m_1+\sum_{i=1}^{n+1}k_i}\circ f$
 is c-equivalent to $T^{m_2}\circ f\circ T_1^{k_1}\circ\cdots\circ T_{n+1}^{k_{n+1}}$ via $(id, \tilde{\psi})$. Compare the twist times of the annuli $A_1,\cdots, A_{n+1}$, under the actions of 
 $T^{m_1+\sum_{i=1}^{n+1}k_i}\circ f$   and  $T^{m_2}\circ f\circ T_1^{k_1}\circ\cdots\circ T_{n+1}^{k_{n+1}}$, we have the following equations
 $$m_2+d_jk_j=m_1+\sum_{i=1}^{n+1}k_i, \  \ 1\leq j\leq n+1.$$
 Therefore $d_j$ is a divisor of $m_1-m_2+\sum_{i=1}^{n+1}k_i$. It follows that ${\rm lcm}(d_1,\cdots, d_{n+1})$ is a divisor of $m_1-m_2+\sum_{i=1}^{n+1}k_i$.
 From these  equations, we get 
 $$\sum_{i=1}^{n+1}k_i=\frac{\sum_{i=1}^{n+1}\frac{1}{d_i}}{1-\sum_{i=1}^{n+1}\frac{1}{d_i}}(m_1-m_2).$$
 So ${\rm lcm}(d_1,\cdots, d_{n+1})$ is a divisor of $(m_1-m_2)/(1-\sum_{i=1}^{n+1}\frac{1}{d_i})$. 
 Put another way, $N^*$ is a divisor of $m_1-m_2$.
 This is a contradiction since $0<|m_1-m_2|<N^*$.
 \end{proof}

  \begin{lem} \label{cui-tan} For any  $0\leq k<N^*$, the map  $T^{k}\circ f$ is  c-equivalent a rational map $R$, unique up to M\"obius conjugation.
 \end{lem}

\begin{proof} This is exactly a special case of Cui-Tan's Theorem, see \cite[Section 6.2 and Lemma 6.2]{CT}.
\end{proof}

Now everything is ready to prove Theorem \ref{covering}.

\vspace{6 pt}

\noindent {\it Proof of Theorem \ref{covering}.}  Take any two marked maps $(f_1, \nu_1),  (f_2, \nu_2)\in \rho^{-1}(\boldsymbol{m})$, 
by the above lemmas, each $(f_j, \nu_j)$ is   c-equivalent to some $T^{k_j}\circ f$ with $0\leq k_j<N^*$. 
The assumptions (C1,C2,C3) on $\boldsymbol{m}$ imply that either $(f_1, \nu_1)=(f_2, \nu_2)$ or $\langle f_1\rangle \neq\langle f_2\rangle$. In the latter case,  it follows 
that $T^{k_1}\circ f$ and
$T^{k_2}\circ f$ are not c-equivalent, implying that $k_1\neq k_2$. 
Hence ${\rm deg}(\rho)\leq N^*$.

On the other hand, for any  $0\leq k<N^*$, by Lemma \ref{cui-tan}, the map $T^k\circ f$ is c-equivalent to a rational map, say $h$.  
By Step 3 in the proof of Proposition \ref{surj},  there  is $h_0\in \langle h\rangle$  with a boundary marking $\nu_0$, so that $(h_0, \nu_0)\in \mathcal{F}^{\chi_0}_{\sigma,\boldsymbol{d}}$ and $\rho((h_0, \nu_0))=\boldsymbol{m}$. Again by the assumptions (C1,C2,C3),  such pair $(h_0, \nu_0)$ is unique, implying that ${\rm deg}(\rho)= N^*$.  \hfill\fbox

\vspace{4 pt}

As an immediate consequence of Theorem \ref{covering}, we have
\begin{cor}\label{onto} The map $\rho:  \mathcal{F}^{\chi_0}_{\sigma,\boldsymbol{d}}\rightarrow  \mathbf{M}_\sigma$ is a homeomorphism if and only if  
$$\sum_{j=1}^{n+1}\frac{1}{d_j}+\frac{1}{{\rm lcm}(d_1, \cdots, d_{n+1})}=1.$$ 
\end{cor}
Note that in this case, the space  $\mathcal{F}^{\chi_0}_{\sigma,\boldsymbol{d}}$ is connected. 
There are various combinations of $d_k$'s  realizing this case. For example, if $n=1$, one may take $\{d_1, d_2\}=\{2,3\}$; if $n=2$, one may set $\{d_1, d_2, d_3\}=\{2,3,7\}, \{3,3,4\},$ etc.

\vspace{5pt}


At the end of this section, we pose the following problem:

\begin{que}\label{connectivity}  What is the action of the covering transformation group of $\rho$ on (each component of) $\mathcal{F}^{\chi_0}_{\sigma,\boldsymbol{d}}$? \  Is $\mathcal{F}^{\chi_0}_{\sigma,\boldsymbol{d}}$ always connected?
\end{que} 


  \section{Global topology of marked hyperbolic component} \label{top_marked}

In the section, we  show

%
%

\begin{thm}\label{marked-comp}  Let $\mathcal{F}$ be any  connected component of  $\mathcal{F}^{\chi_0}_{\sigma,\boldsymbol{d}}$,  then $\mathcal{F}$ is homeomorphic to $\mathbb{R}^{4d-4-n}\times \mathbb{T}^n.$
\end{thm}

\begin{proof} The proof is a pure algebraic topology argument.  First, it is known from Theorems \ref{covering-p} and \ref{covering} that $\rho:  \mathcal{F}^{\chi_0}_{\sigma,\boldsymbol{d}}\rightarrow \mathbf{M}_\sigma$  is a finite-to-one  covering map.  In particular, the proof of  Theorem \ref{covering-p} implies that both $\rho(\mathcal{F})$ and $\mathbf{M}_\sigma\setminus\rho(\mathcal{F})$ are open in $\mathbf{M}_\sigma$. 
 Therefore $\rho(\mathcal{F})=\mathbf{M}_\sigma$, and the restriction of $\rho$ on $\mathcal{F}$, still denoted by $\rho$,  is  a finite-to-one covering map.
Chose a base point $\tau_0=(f_0, \nu_0)\in  \mathcal{F}$, let  $\rho(\tau_0)=\boldsymbol{m}_0=(m_0, m_\infty, m_1, \cdots, m_n)\in \mathbf{M}_\sigma$  be its model map.
By  \cite[Proposition 1.31]{H}, we know that $\rho$ induces  an  injective group homomorphism  
$$\rho_*: 
\begin{cases}
\pi_1({\mathcal{F}},\tau_0)\rightarrow \pi_1(\mathbf{M}_\sigma, \boldsymbol{m}_0), \\
 [\gamma]\mapsto [\rho(\gamma)] ,
\end{cases}$$
where $\gamma$ is a  loop in $\mathcal{F}$ starting and ending at $\tau_0$, and the notation  $[\beta]$ means the homotopy class of the loop $\beta$ in the underlying topological space.

Recall that   $\mathbf{M}_\sigma$ is the Cartesian product of the spaces $ \boldsymbol{B}^{\rm{fc}}_{d_1}$, $\boldsymbol{B}^{\rm{fc}}_{d_{n+1}}$ (or $\boldsymbol{B}^{\rm{zc}}_{d_{n+1}}$), and  $\boldsymbol{A}(d_k+d_{k+1}, \delta_k),  1\leq k \leq n$.   By Theorem \ref{model-space}, for each $1\leq k\leq n$, we may find a loop 
 $\gamma_k $  in $\boldsymbol{A}(d_k+d_{k+1}, \delta_{k})$ with base point $m_k$, representing  the generator  of the fundamental group. 
 The homotopy class  $[\gamma_k]$  induces a  homotopy class   $[\gamma^*_k]$ in $\mathbf{M}_\sigma$. Here the loop $\gamma^*_k: S^1\rightarrow \mathbf{M}_\sigma$ is  defined by
 $$\gamma^*_k(t)=(m_0, m_\infty, m_1,\cdots, m_{k-1},  \gamma_k(t), m_{k+1},  \cdots, m_n), \ t\in S^1.$$
 
The fundamental group $\pi_1(\mathbf{M}_\sigma, \boldsymbol{m}_0)$ of $\mathbf{M}_\sigma$ is a free   Abelian group  with $n$ generators, and it can be expressed  as the direct product
$$\pi_1(\mathbf{M}_\sigma, \boldsymbol{m}_0)= [\gamma^*_1]\mathbb{Z}\times \cdots \times  [\gamma^*_n]\mathbb{Z}.$$

Since $\rho:\mathcal{F}\rightarrow \mathbf{M}_\sigma$ is  finite-to-one, each loop $\gamma^*_k$ has finite order in ${\mathcal{F}}$. That is, there is a positive integer,
denoted  by  ${\rm ord}(\gamma^*_k)$, such that ${\rm ord}(\gamma^*_k)\gamma^*_k$ lifts to a loop in $\mathcal{F}$.
Therefore we have 
$${\rm ord}(\gamma^*_1)[\gamma^*_1] \mathbb{Z} \times \cdots \times {\rm ord}(\gamma^*_n)[\gamma^*_n]  \mathbb{Z}\subset \rho_*(\pi_1(\mathcal{F},\tau_0)).$$
This implies that the Abelian subgroup $ \rho_*(\pi_1(\mathcal{F},\tau_0))$  has rank $n$. By  the structure theorem of Abelian groups (see \cite[Theorem 3, page 158]{DF}, there is a basis of 
homotopy classes of 
loops
$$[\alpha_1], \cdots, [\alpha_n] \in \rho_*(\pi_1(\mathcal{F},\tau_0)).$$ 
so that 
$$\rho_*(\pi_1(\mathcal{F},\tau_0))=[\alpha_1]\mathbb{Z}\times \cdots \times  [\alpha_n]\mathbb{Z}. $$

Suppose that for each $1\leq i\leq n$,   
$$[\alpha_i]= \sum_{j=1}^n m_{ij}[\gamma^*_j],$$
where $m_{ij}$ are integers.
The matrix $B=(m_{ij})$ is reversible.  (More precisely, we remark  that ${\rm deg}( \rho|_{\mathcal{F}})=|{\rm det}(B)|$. )

 Now, we define a finite-to-one covering map 
 $$L: 
\begin{cases}\mathbb{R}^{4d-4-n}\times \mathbb{T}^n \rightarrow \mathbb{R}^{4d-4-n}\times \mathbb{T}^n , \\
(X, Y)\mapsto(X, B^tY)
\end{cases}$$
 where $X\in  \mathbb{R}^{4d-4-n}$ and $Y\in \mathbb{T}^n$.  Suppose that $P_0$ is a $L$-preimage of $\Phi(\boldsymbol{m}_0)$, and let $\Phi: \mathbf{M}_\sigma \rightarrow \mathbb{R}^{4d-4-n}\times \mathbb{T}^n$  be the homeomorphism  in Corollary \ref{model-space1}.   By the definition of $L$, we know that 
 $$L_*(\pi_1( \mathbb{R}^{4d-4-n}\times \mathbb{T}^n, P_0))=(\Phi\circ\rho)_*(\pi_1(\mathcal{F},\tau_0)).$$
 By the  lifting criterion  \cite[Proposition 1.37]{H}, there is  a unique  homeomorphism $\Psi: \mathcal{F}\rightarrow \mathbb{R}^{4d-4-n}\times \mathbb{T}^n$ with 
 $\Psi(\tau_0)=P_0$. See the following commutative diagram 
  
$$
\xymatrix{ & ({\mathcal{F}}, \tau_0) \ar[r]^{\Psi \ \ \ \  \ \ \ \  }
\ar[d]_{\rho}
& (\mathbb{R}^{4d-4-n}\times \mathbb{T}^n,  P_0)  \ar[d]^{L }\\
&(\mathbf{M}_\sigma, \boldsymbol{m}_0)  \ar[r]_{\Phi \ \ \ \  \ \ \ \ } & (\mathbb{R}^{4d-4-n}\times \mathbb{T}^n,  \Phi(\boldsymbol{m}_0))  }
$$ 
The proof is completed.
\end{proof}

 \section{Global topology of  hyperbolic component} \label{global-hyp}

 This section gives the proof of Theorem \ref{cantor-circle-locus}, restated as follows 
 
 \begin{thm}\label{glob-comp} There is a finite-to-one quotient map $\boldsymbol{q}:\mathbb{R}^{4d-4-n}\times \mathbb{T}^n \rightarrow \mathcal{H}$. In other words, $\mathcal{H}$ is homeomorphic to the quotient space $\mathbb{R}^{4d-4-n}\times \mathbb{T}^n/{\boldsymbol{q}}$, here points $\boldsymbol{x}_1, \boldsymbol{x}_2\in \mathbb{R}^{4d-4-n}\times \mathbb{T}^n$
 are equivalent if $\boldsymbol{q}(\boldsymbol{x}_1)=\boldsymbol{q}(\boldsymbol{x}_2)$.
\end{thm}
 
Note that there is a natural projection from $\mathcal{F}^{\chi_0}_{\sigma,\boldsymbol{d}}$ to  ${M}_d$
$$\boldsymbol{p}:  \mathcal{F}^{\chi_0}_{\sigma,\boldsymbol{d}} \rightarrow M_d, \ \ \ \  (f,\nu)\mapsto \langle f\rangle.$$
Let $\mathcal{H}_{\sigma,\boldsymbol{d}}=\boldsymbol{p}(\mathcal{F}^{\chi_0}_{\sigma,\boldsymbol{d}})$ be $\boldsymbol{p}$'s   image. Clearly,  $\mathcal{H}_{\sigma,\boldsymbol{d}}$ contains $\mathcal{H}$ as a component.

 Let $\mathcal{F}$ be a  component of $\mathcal{F}^{\chi_0}_{\sigma,\boldsymbol{d}}$  so
  that $\boldsymbol{p}(\mathcal{F})=\mathcal{H}$. 
 To study the topology of $\mathcal{H}$, we need decompose $\boldsymbol{p}|_{\mathcal{F}}$ into two parts
$$\boldsymbol{p}_1: 
\begin{cases}
\mathcal{F}  \rightarrow\mathcal{E}\\
(f, \nu)\mapsto f
\end{cases}   \text{ and   }  \   \boldsymbol{p}_2: 
\begin{cases}
\mathcal{E}\rightarrow \mathcal{H}\\
f\mapsto \langle f\rangle
\end{cases}$$
where $\mathcal{E}=\boldsymbol{p}_1(\mathcal{F})$, which is necessarily a component of $\mathcal{F}_{\sigma,\boldsymbol{d}}$. 
Clearly, both $\boldsymbol{p}_1$ and $\boldsymbol{p}_2$ are continuous.
Since every map $f\in \mathcal{E}$ has only finitely many boundary markings $\nu$ with 
 $\chi(\nu)=\chi_0$, one may  establish

  \begin{lem} \label{cov1} The map $\boldsymbol{p}_1: \mathcal{F}\rightarrow \mathcal{E}$ is a finite-to-one covering map, whose   covering transformation group is Abelian.
 \end{lem}
%
 
 \begin{proof} The proof of the finite covering property of $\boldsymbol{p}_1$ is easy.  We only prove that the covering  transformation group is  Abelian.
 Fix a map $f\in \mathcal{E}$, and let $B_f=\{\nu: (f, \nu)\in \mathcal{F}\}$. 
Recall that $\{D_k\}_{k\in \mathcal{I}}$  are the critical Fatou components of $f$ (see Section \ref{dynamics}). 
 For each $k\in\{1,\cdots, n, \infty\}$, let $s_{k}^{\sigma}$ be the number of all   possible candidates of the marked point  $\nu(D_k)$.  
 It's not hard to see that  (here $1\leq j\leq n$)
 $$s_{\infty}^{I}=d_{n+1}, \ s_{\infty}^{II}=1, \ s_{\infty}^{III}=d_{n+1}-1, \ s_{j}^{II}=d_j+d_{j+1}-\delta_j,$$
 $$s_{j}^{\sigma}=d_j+d_{j+1}-\delta_j  (j \text{ even}) \text{ or } 
 (d_j+d_{j+1}-\delta_j)s^{\sigma}_{\infty} (j \text{ odd}), \ \sigma=I,III.$$
 The possible
  candidates of the marked point on the corresponding boundary component of $D_k$ are labeled in positive cyclic order by $0,1,\cdots, s_{k}^{\sigma}-1$.  In this way, we  get an injective map
 $$\eta:\begin{cases} B_f \rightarrow G_\sigma=\mathbb{Z}_{s_{1}^{\sigma}}\times \cdots\times \mathbb{Z}_{s_{n}^{\sigma}}\times  \mathbb{Z}_{s_{\infty}^{\sigma}}\\
 \nu\mapsto(i_1, \cdots, i_n, i_{n+1})\end{cases}$$
 where $i_k$ is the number so that $\nu(D_k)$ is the $i_k$-th marked point, for $k\in\{1,\cdots, n, \infty\}$,  and $\mathbb{Z}_s$ is the cyclic group of order $s\geq1$.
 
 It is important to observe that  every loop class $[\gamma]\in \pi_1(\mathcal{E}, f)$ 
 induces  an action on $B_f$ and an element $Z_{[\gamma]}\in G_\sigma$
 so that
 $$\eta([\gamma]\cdot \nu)= \eta(\nu)+Z_{[\gamma]}, \ \forall \ \nu\in B_f.$$
 In fact this $[\gamma]\cdot \nu \in B_f$ is defined so that the lift $\widetilde{\gamma}\subset\mathcal{F}$ of $\gamma$  starts at $(f, \nu)$ and 
 ends at $(f, [\gamma]\cdot \nu)$.
By the following fact
 \bess\eta([\gamma_1][\gamma_2]\cdot \nu)&=&\eta([\gamma_1]\cdot ([\gamma_2]\cdot \nu))=  
 \eta(\nu)+Z_{[\gamma_1]}+Z_{[\gamma_2]}\\&=&\eta([\gamma_2]\cdot ([\gamma_1]\cdot \nu))=
 \eta([\gamma_2][\gamma_1]\cdot \nu),\eess
we see that the covering  transformation group is  Abelian.
  \end{proof}

  \begin{lem} \label{bcov2} The map $\boldsymbol{p}_2: \mathcal{E}\rightarrow \mathcal{H}$ is a finite-to-one (possibly branched) covering map. 
 %
%
%
%
%
 \end{lem}

\begin{proof}  Let $f\in \mathcal{E}$ and 
  $a^f_0=1, a^f_1, \cdots, a^f_{s_0^\sigma-1}$ be all fixed points of $f$ (if $\sigma=I, III$) or $f^{2}$ (if $\sigma=II$) on $\partial D_0$, in positive cyclic order, where $s_0^I=s_0^{III}=d_1-1, \ s_0^{II}=d_1d_{n+1}-1$.
 Note that 
 for any $\langle f\rangle\in  \mathcal{H}$, 
$$\boldsymbol{p}_2^{-1}(\langle f\rangle)=\begin{cases}
 \{ f(a_k^f\cdot)/a_k^f;   k\}\cap\mathcal{E}, & \sigma=I,\\
 \{ f(a_k^f\cdot)/a_k^f, \ b_j^f/f(b_j^f/\cdot);   k, j\}\cap\mathcal{E}, & \sigma=II \text{ or } III.
\end{cases}$$
where $b_j$'s are the fixed points of $f$ (if $\sigma=III$) or $f^2$ (if $\sigma=II$) on $\partial D_\infty$.  

Therefore the cardinality $|\boldsymbol{p}_2^{-1}(\langle f\rangle)|\leq s_0^\sigma$ if $\sigma=I$, and  $|\boldsymbol{p}_2^{-1}(\langle f\rangle)|\leq 2 s_0^\sigma$ if $\sigma=II$ or $III$,  implying that  $\boldsymbol{p}_2$ is finite-to-one.

Let ${\rm Aut}(f)$ be the automorphism group of $f$ (see Section \ref{aut-rat} for precise definition).
One may observe that $\boldsymbol{p}_2$ is a covering  map if and only if  ${\rm Aut}(f)$ is trivial for
every map $f\in\mathcal{E}$. More generally, 
the branched set of $\boldsymbol{p_2}$ is exactly $\mathcal{B}(\boldsymbol{p}_2)=\{f\in \mathcal{E}; {\rm Aut}(f) \text{ is nontrivial}\}$,  
and $\boldsymbol{p_2}$ is a covering map from $\mathcal{E}\setminus \boldsymbol{p_2}^{-1}(\boldsymbol{p_2}(\mathcal{B}(\boldsymbol{p}_2)))$ onto $\mathcal{H}\setminus \boldsymbol{p_2}(\mathcal{B}(\boldsymbol{p}_2))$.
\end{proof}

To further understand the relation between  $\mathcal{E}$ and $\mathcal{H}$, let's define 
the following group (analogous to the covering transformation group)
 $$ G(\boldsymbol{p}_2)=\{\kappa: \mathcal{E}\rightarrow \mathcal{E} \text{ is a  homeomorphism, and }\boldsymbol{p}_2\circ\kappa=\boldsymbol{p}_2\}.$$
 Note that the maps
  $$\zeta_k: 
\begin{cases}\mathcal{E}\rightarrow \mathbb{C}, \\
f\mapsto a_k^f
\end{cases}(\text{for any } \sigma), \ \  \zeta^*_k: 
\begin{cases}\mathcal{E}\rightarrow \mathbb{C}, \\
f\mapsto f(a_k^f)
\end{cases}(\text{for } \sigma=II) $$
  are well-defined. Therefore
if $\sigma=I$, then
$$ G(\boldsymbol{p}_2)=\{f\mapsto f(a_k^f\cdot)/a_k^f;  f(a_k^f\cdot)/a_k^f\in\mathcal{E} \}$$
is cyclic; if $\sigma=II$, then
$$ G(\boldsymbol{p}_2)=\{f\mapsto f(a_k^f\cdot)/a_k^f \text{ or }  f(a_k^f)/f(f(a_k^f)/\cdot) ;  f(a_k^f\cdot)/a_k^f, f(a_k^f)/f(f(a_k^f)/\cdot) \in\mathcal{E} \}$$
is either cyclic or dihedral.

In these two cases,   $\mathcal{H}$ can be identified as the quotient space  $\mathcal{E}/G(\boldsymbol{p}_2)$. 
 Here are some examples that one can write down $G(\boldsymbol{p}_2)$ explicitly. All of them satisfy the technical assumption  (see  Corollary \ref{onto})
$$\sum_{k=1}^{n+1}\frac{1}{d_k}+\frac{1}{{\rm lcm}(d_1, \cdots, d_{n+1})}=1$$ 
to guarantee that both $\mathcal{F}^{\chi_0}_{\sigma,\boldsymbol{d}}$ and $\mathcal{F}_{\sigma,\boldsymbol{d}}$ are connected. Therefore $\mathcal{E}=\mathcal{F}_{\sigma,\boldsymbol{d}}$.
 
\textbf{$G(\boldsymbol{p}_2)$ is trivial:}  Let $n=1$,  $\boldsymbol{d}=(d_1,d_2)=(2,3)$.  The map  $\boldsymbol{p}_2: \mathcal{E}\rightarrow \mathcal{H}$ is a homeomorphism (because ${\rm Aut}(f)$ is trivial for all $f\in \mathcal{E}$, by Proposition \ref{aut}),  and $G(\boldsymbol{p}_2)$ is trivial.

\textbf{$G(\boldsymbol{p}_2)$ is cyclic:}  Let $n=2$,  $\boldsymbol{d}=(d_1,d_2,d_3)=(3,3,4)$  and $\sigma=II$. The map $\boldsymbol{p}_2: \mathcal{E}\rightarrow \mathcal{H}$ is a 11-to-1 covering map (because every map $f\in \mathcal{E}$ has a trivial automorphism group, by Proposition \ref{aut}), and $G(\boldsymbol{p}_2)\cong \mathbb{Z}_{11}$.

\textbf{$G(\boldsymbol{p}_2)$ is dihedral:}   Let $n=2$, $\boldsymbol{d}=(d_1,d_2,d_3)=(6,2,6)$ and $\sigma=II$. The map $\boldsymbol{p}_2: \mathcal{E}\rightarrow \mathcal{H}$ is 70-to-1 branched cover, and $G(\boldsymbol{p}_2)\cong D_{70}$.

\vspace{8pt}

\noindent {\it Proof of Theorem \ref{glob-comp}.}  It follows from Lemmas \ref{cov1} and \ref{bcov2} that 
$\boldsymbol{p}: \mathcal{F}\rightarrow \mathcal{H}$ is a finite-to-one quotient mapping. Identifying $\mathcal{F}$ 
and $\mathbb{R}^{4d-4-n}\times \mathbb{T}^n$ via $\Psi$  (by Theorem \ref{marked-comp}), we see that 
$\mathcal{H}$ is homeomorphic to the  quotient space 
$$\mathbb{R}^{4d-4-n}\times \mathbb{T}^n/\boldsymbol{q},$$
with $\boldsymbol{q}=\boldsymbol{p}\circ \Psi^{-1}$. \hfill\fbox

\begin{rmk} An arithmetic condition imposed on $\boldsymbol{d}$ can guarantee that the quotient map $\boldsymbol{q}$ is
 a covering map, see Corollary \ref{cover}. 
\end{rmk}

\section{Appendix} \label{appendix}

This section provides the supplementary materials to the paper, including:

\begin{itemize}
 \item
  a proof of unboundedness of hyperbolic component; 
  \item
 automorphism group of rational maps in Cantor circle locus;
\item
topological space finitely covered by $\mathbb{T}^n$;
 \item
crystallographic group.
 \end{itemize}

\subsection{Unboundedness of hyperbolic components}

This part gives an alternative proof of a result due to Makienko \cite{Ma}.

\begin{thm}\label{unb} Let  $\mathcal{H}\subset{M}_d$ be a hyperbolic component  in the  disconnectedness locus, then
$\mathcal{H}$ is unbounded.
\end{thm}

\begin{proof} Suppose $\langle f\rangle\in \mathcal{H}$ and  $f$ has $N$ attracting cycles $C_1(f), \cdots, C_N(f)$, with multiplier vector $\mathbf{\lambda}(f)=(\lambda_1,\cdots, \lambda_k, 0,\cdots,0)$,
where $\lambda_j\neq0$ is the multiplier of $f$ at the cycle $C_j(f), j=1,\cdots,k$ , and the cycles $C_j(f)$ with $k<j\leq N$ are superattracting.

Applying quasiconformal surgery, we get a continuous  family of rational maps $f_t$ with $\mathbf{\lambda}(f_t)=t\mathbf{\lambda}(f), 0<t\leq1$.
If $\langle f_t\rangle$ has no accumulation point in ${M}_d$ as $t\rightarrow0$, then $\mathcal{H}$ is unbounded.  Else, let $\langle g\rangle\in {M}_d$ be an accumulation point. {Then $g$ is hyperbolic with $N$ superattracting cycles}.

Let $G_g^j$ be the Green function of $g$ at the cycle $C_j(g)$, and $A_j$ be the whole attracting basin of $C_j(g)$.  Let $\phi_\tau$ solve the Beltrami equation:
$$\frac{\overline{\partial }\phi_\tau}{\partial \phi_\tau}=\frac{\tau-1}{\tau+1}
\sum_{j=1}^N \mathbf{1}_{A_j}\frac{\overline{\partial }G_g^j}{\partial G_g^j},\  1\leq \tau<+\infty.$$

We consider the curve $\langle g_\tau\rangle_{1\leq \tau<+\infty}$ in $\mathcal{H}$, where $g_\tau=\phi_\tau\circ g\circ \phi_\tau^{-1}$. If $\langle g_\tau\rangle$ has no accumulation point in ${M}_d$, the $\mathcal{H}$ is unbounded. Else, let $\langle g_\infty\rangle\in {M}_d$ be the limit point of some sequence $\langle g_{\tau_n}\rangle$ with
$1\leq \tau_1<\tau_2<\cdots \rightarrow\infty$. By suitably choosing representatives, we may assume $g_{\tau_n}$ converges to $g_\infty$ in $\mathbb{\widehat{C}}$.
 Then $g_\infty$ is hyperbolic, thus $J(g_\infty)$ is disconnected. On the other hand, since $g_{\tau_n}$ is the streching sequence of $g$, the map
 $g_\infty$ is necessarily postcritically finite (an interesting fact). So $J(g_\infty)$ is connected. Contradiction.
\end{proof}

The reverse of Theorem \ref{unb} is not true. An unbounded hyperbolic component can sit in the connectedness locus, see \cite{E}.

\subsection{Automorphism group of rational maps} \label{aut-rat}
 
 Recall that the automorphism group of a  rational map $f$ is defined by
 $${\rm Aut}(f)=\{\phi\in PSL(2, \mathbb{C}); \phi \circ f=f\circ \phi\}.$$
  One may view the trivial group as $\mathbb{Z}_1$, the cyclic group of order 1.
 It is known from \cite{MSW,Si} that   ${\rm Aut}(f)$ 
 is isomorphic to 
  \begin{itemize}
 \item
  a cyclic group $\mathbb{Z}_m$ of order $m\geq1$, or 
   \item
 the dihedral group $D_{2m}$ of order $2m$, or 
 \item
 the alternating groups $A_4, A_5$,  or
 \item
 the symmetric group $S_4$.
 \end{itemize}
 However, for hyperbolic rational map $f$ with Cantor circle Julia set, there are only two possibilities for ${\rm Aut}(f)$:  cyclic or dihedral. Precisely,
 

  \begin{pro} \label{aut} Let $f\in \mathcal{F}_{\sigma,\boldsymbol{d}}$, here $\boldsymbol{d}=(d_1,d_2,\cdots,d_{n+1})$.
  
  1. If $\sigma=I$, then  ${\rm Aut}(f)$ is isomorphic to $\mathbb{Z}_\ell$, where
  $$\ell|{\rm gcd}\{d_k+(-1)^k; 1\leq k\leq n+1\}.$$

  2. If  $\sigma=II$, then  ${\rm Aut}(f)$ is isomorphic to  $\mathbb{Z}_\ell$ or $D_{2\ell}$, where
  $$\ell|{\rm gcd}\{d_k-(-1)^k; 1\leq k\leq n+1\}.$$
  
   3. If $\sigma=III$, then  ${\rm Aut}(f)$ is isomorphic to $\mathbb{Z}_\ell$ or $D_{2\ell}$, where
  $$\ell|{\rm gcd}\{d_k+(-1)^k; 1\leq k\leq n+1\}.$$
 \end{pro}
 
%
 
 To prove Lemma \ref{aut}, we need a fact  on the model maps.
 For $m\in \boldsymbol{B}^{\rm{fc}}_{D}\cup\boldsymbol{B}^{\rm{zc}}_{D}\cup\boldsymbol{A}(e,  \delta)$, let's define 
 $${\rm Rot}^{\pm}(m)=\{\theta\in S^1; m(e^{2\pi i \theta}z)=e^{{\pm}2\pi i \theta}m(z),  \forall z \}.$$
Clearly, ${\rm Rot}^+(m)$  characterizes the rotation symmetries of $m$.

  \begin{lem}\label{model-aut} 1. For any $m\in \boldsymbol{B}^{\rm{fc}}_{D}\cup\boldsymbol{B}^{\rm{zc}}_{D}$,  one has 
  $${\rm Rot}^{\pm}(m)=\{k/\ell; 0\leq k<\ell\} \text{ with }\ell|D\mp 1.$$
  
  2.  For any  $m\in \boldsymbol{A}(e,  \delta)$, one has    
  $${\rm Rot}^{\pm}(m)=\{k/\ell; 0\leq k<\ell\} \text{ with }\ell|{\rm gcd}(\delta {\pm}1, e).$$  
  \end{lem}
 
 \begin{proof} 1. Take $\theta\in {\rm Rot}^{+}(m)$, note that $r_\theta(z)=e^{2\pi i \theta}z$ maps the fixed point set of $m$ on $\partial \mathbb{D}$ 
bijectively  to itself, and there  are exactly $D-1$ fixed points on $\partial \mathbb{D}$. 
 So 
 each element $\theta\in {\rm Rot}(m)$ necessarily satisfies $(D-1)\theta\in\mathbb{Z}$. Since ${\rm Rot}(m)$ is a finite abelian group of numbers, it must take the form $\{k/\ell; 0\leq k<\ell\} \text{ with }\ell|D-1$.
 
 If  $\theta\in {\rm Rot}^{-}(m)$, then $r_\theta$ commutes with $1/m$ on $\partial \mathbb{D}$.  Note that $r_\theta$ maps the fixed point set of $1/m$ on $\partial \mathbb{D}$ 
bijectively  to itself, and $1/m$ has exactly $D+1$ fixed points on $\partial \mathbb{D}$. The rest arguments are the same as above.
 
 
 2. If $\theta\in {\rm Rot}^{+}(m)$, then by the same reason as 1,   any rotation $r$ commuting with $m$ takes the form $r(z)=e^{2\pi i k/\ell_0}z$ with  $ \ell_0| e-\delta-1$, ${\rm gcd}(k, \ell_0)=1$ and $0\leq k<\ell_0\geq 1$. The zero set $m^{-1}(0)$ is invariant under $r$ implying that $\ell_0| e$. Therefore  $\ell_0|{\rm gcd}(\delta+1, e)$.
 Since ${\rm Rot}^+(m)$ is a finite abelian group of numbers, it must take the form $\{k/\ell; 0\leq k<\ell\} \text{ with }\ell|{\rm gcd}(\delta+1, e)$.
 
  If $\theta\in {\rm Rot}^{-}(m)$, assume $m$ is defined on $\mathbb{A}_r$. Take $a\in m^{-1}(1)\cap C_r$ and
   let $\widehat{m}(z)=m(a/z)$, then $\widehat{m}\in\boldsymbol{A}(e,  e-\delta)$ and $-\theta\in {\rm Rot}^{+}(\widehat{m})$. The rest arguments are the  same as above.
   \end{proof}

   \begin{rmk}  Every set $\{k/\ell; 0\leq k<\ell\}$ with $\ell\geq 1$ and  $\ell|D\mp 1$ (first case) or $\ell|{\rm gcd}(\delta\pm 1, e)$ (second case) is realizable by ${\rm Rot}^{\pm}(m)$ for some $m$.
   \end{rmk}
 
 \vspace{5pt}
 
 \noindent {\it Proof of Proposition \ref{aut}.} 
 Assume that  ${\rm Aut}(f)$ is nontrivial.  Take   $\phi\in {\rm Aut}(f)\setminus\{id\}$. Necessarily $\phi(\{0,\infty\})=\{0,\infty\}$, and two cases happen:
  
\textbf{Case 1 (rotation):}  $\phi$ fixes $0$ and $\infty$. In this case $\phi(z)=a z,  a\neq 0,1$. The relation $\phi\circ f=f\circ \phi$ implies that  $\phi( \partial D_0)= \partial D_0$  and $a^j$ is a fixed point of $f$ on $\partial D_0$, 
 for all $j\geq 1$.
 Therefore there is a minimal integer $\ell_0>1$ so that  $\phi(z)=e^{2\pi i k_0/\ell_0}z$ with ${\rm gcd}(k_0, \ell_0)=1$.
 
 If $\sigma=I$, we choose a boundary marking $\nu$ with $\chi(\nu)=(-,+,-,+,\cdots)$. 
 Then $\phi\circ f=f\circ \phi$ implies that 
   $\boldsymbol{m}((f,\nu))=(m_0,m_\infty, m_1,\cdots,m_n)$ satisfies
 $$k_0/\ell_0\in{\rm Rot}^+(m_0)\cap {\rm Rot}^+(m_2)\cap {\rm Rot}^+(m_4)\cap \cdots,$$
 $$-k_0/\ell_0\in{\rm Rot}^-(m_\infty)\cap {\rm Rot}^+(m_1)\cap {\rm Rot}^+(m_3)\cap\cdots.$$
 
 By Lemma \ref{model-aut}, we get
 $$\ell_0|{\rm gcd}\{d_k+(-1)^k; 1\leq k\leq n+1\}.$$
 
 Similarly, for $\sigma=II$, 
 $$\ell_0|{\rm gcd}\{d_k-(-1)^k; 1\leq k\leq n+1\},$$
  and for  $\sigma=III$,
 $$\ell_0|{\rm gcd}\{d_k+(-1)^k; 1\leq k\leq n+1\}.$$
%
 
 
   
 \textbf{Case 2 (involution):}  $\phi$ interchanges $0$ and $\infty$. In this case $\phi(z)=b/z$ with $b\neq 0$. Clearly this case happens only if  $\sigma=II \text{ or } III$  (implying that $n$ is even)  and $d_k=d_{n+2-k},  \forall \ 1\leq k\leq n/2$.

\vspace{5pt}

Now let's look at ${\rm Aut}(f)$.  

 If ${\rm Aut}(f)\setminus\{id\}$ consists of rotations, then we  have ${\rm Aut}(f)\cong \mathbb{Z}_\ell$  for some $\ell|{\rm gcd}\{d_k+(-1)^k; 1\leq k\leq n+1\}$   (if $\sigma=I, III$) or 
  $\ell|{\rm gcd}\{d_k-(-1)^k; 1\leq k\leq n+1\}$
  (if $\sigma=II$).
 
  If ${\rm Aut}(f)\setminus \{id\}$ consists of involutions, then take any two maps  $\phi_1(z)=b_1/z, \phi_2(z)=b_2/z$ in   ${\rm Aut}(f)$, we have $\phi_1\circ\phi_2 \in {\rm Aut}(f)$, implying that $b_1=b_2$. So ${\rm Aut}(f)$ consists of $id$ and $\phi_1$,  hence isomorphic to $\mathbb{Z}_2=D_2$.
  
 If  ${\rm Aut}(f)\setminus \{id\}$ contains a rotation and an involution, then ${\rm Aut}(f)$ is generated by 
 an involution $z\mapsto b/z$ and a primitive rotation $z\mapsto e^{2\pi i k/\ell}z$. In this case, ${\rm Aut}(f)$ is isomorphic to the dihedral group $D_{2\ell}$.
\hfill\fbox

%
%
%
%
%
%
%
%

  \begin{cor}\label{cover}  Let $\boldsymbol{d}^*=(d_{n+1},\cdots,d_{1})$.  Assume that
  
  ${\rm gcd}\{d_k+(-1)^k; 1\leq k\leq n+1\}=1$ (when $\sigma=I$);
  
   ${\rm gcd}\{d_k-(-1)^k; 1\leq k\leq n+1\}=1$ and $\boldsymbol{d}\neq \boldsymbol{d}^*$ (when $\sigma=II$);
  
   ${\rm gcd}\{d_k+(-1)^k; 1\leq k\leq n+1\}=1$ and $\boldsymbol{d}\neq \boldsymbol{d}^*$  (when $\sigma=III$).

\noindent Then $\mathcal{H}$ is finitely covered by $\mathbb{R}^{4d-4-n}\times \mathbb{T}^n$.
 \end{cor}
 \begin{proof}  Recall the definitions of $\boldsymbol{p}_1:\mathcal{F}\rightarrow\mathcal{E}$ and $\boldsymbol{p}_2: \mathcal{E}\rightarrow \mathcal{H}$ in the previous section.  By Proposition \ref{aut}, the arithmetic condition implies that  ${\rm Aut}(f)$ is trivial for each map $f\in  \mathcal{E}$.  Therefore
 $\boldsymbol{p}_2: \mathcal{E}\rightarrow \mathcal{H}$ is a finite-to-one covering map. Consequently, the map 
 $\boldsymbol{q}=\boldsymbol{p}\circ \Psi^{-1}: \mathbb{R}^{4d-4-n}\times \mathbb{T}^n \rightarrow \mathcal{H}$
 defined in the proof of Theorem \ref{glob-comp} is a finite-to-one covering map.
 \end{proof}
 
 Under the assumption as in Corollary \ref{cover}, we are facing the problem: what a space finitely covered by  $\mathbb{R}^{4d-4-n}\times \mathbb{T}^n$ can be?
 If we forget the dynamical setting and  simply  consider the problem from the topological viewpoint,  the classification of spaces finitely covered by  $\mathbb{T}^n$ (a factor of $\mathbb{R}^{4d-4-n}\times \mathbb{T}^n$) is already an important subject in topology,  related to the classifications of  crystallographic  groups and dating back more than 100 years ago to David Hilbert \cite{M2}. There are numerous references on the subject, e.g.   \cite{B1,B2,Ch,F,FG,FH} and an accessible one is \cite{Hi}. We would mention an interesting example in Section \ref{hatcher-eg}.
    
    The classification of spaces finitely covered by  $\mathbb{R}^{4d-4-n}\times \mathbb{T}^n$ is more delicate than that by $\mathbb{T}^n$, even in the case $n=1$. For example, a topological space finitely covered by $\mathbb{T}^1=S^1$ is nothing but $S^1$, while  a topological space finitely covered by $\mathbb{R}^1\times\mathbb{T}^1$ can be 
    a M\"obius band or $\mathbb{R}^1\times\mathbb{T}^1$.

\subsection{Topological space finitely covered by $\mathbb{T}^n$} \label{hatcher-eg}

  In dimension $n=1$ or  2, if an orientable topological space $X$ is finitely covered by $\mathbb{T}^n$, then $X$ is necessarily homeomorphic to $\mathbb{T}^n$. However, in dimension 3, things are different. The following example is provided to us by Allen Hatcher.

  Let $f:\mathbb{T}^2 \rightarrow\mathbb{T}^2$ be the map that is a reflection of each circle factor of $\mathbb{T}^2$.  One can realize $f$ as a certain 180 degree rotation of the torus.  Now let $X$ be the quotient space of $\mathbb{T}^2\times [0,1]$ with the points $(x,0)$ and $(f(x),1)$ identified.  Then $X$ has a 2-sheeted covering space which is the quotient space of $\mathbb{T}^2\times [0,2]$ with $(x,0)$ identified with $(f^2(x),2)$ and this is $\mathbb{T}^3$ since $f^2$ is the identity map.  The covering transformation group is  cyclic of order 2.  The space $X$ is not homeomorphic to $\mathbb{T}^3$ since its fundamental group $\pi_1(X)$ is nonabelian, with presentation
   $$\langle a,b,c \ | \ ab=ba, \ cac^{-1}=a^{-1}, \ cbc^{-1}=b^{-1} \rangle.$$
This group is isomorphic to the  semidirect product 
$$\mathbb{Z}^2\rtimes_\phi \mathbb{Z}$$
where the $\mathbb{Z}$-action $\phi$ on $\mathbb{Z}^2$ is induced by
$$\phi(1)\cdot(1,0)=(-1,0), \ \phi(1)\cdot(0,1)=(0,-1).$$

%
%
%
%
%
%
%
%
%

\subsection{Crystallographic group}  We only give a quick introduction to crystallographic groups, more interesting stuff can be found in \cite{Hi}.
Let $k\in\mathbb{N}$ and ${\rm Isom}(\mathbb{R}^k)$ be the group of all isometries of $\mathbb{R}^k$. It is known that  each element $\phi\in {\rm Isom}(\mathbb{R}^k)$ can be written uniquely as $\phi=T_{\boldsymbol{v}}\circ U$, where $T_{\boldsymbol{v}}(\boldsymbol{x})=\boldsymbol{x}+\boldsymbol{v}$ is a translation and $U$ is an element of the orthogonal group ${\rm O}(k)=\{A; AA^t=I\}$.
A discrete subgroup $G$ of ${\rm Isom}(\mathbb{R}^k)$  is called a $k$-dimensional {\it crystallographic group} (or  {\it space group}) if the quotient space $\mathbb{R}^k/G$ is compact. A purely group-theoretic characterization of  crystallographic group due to Zassenhaus \cite{Z} states that an abstract group $G$ is isomorphic to a $k$-dimensional crystallographic group if and only if $G$ contains an finite index, normal, free-abelian subgroup of rank $k$, that is also maximal abelian.

Under the assumption as in Corollary \ref{cover}, the quotient map $\boldsymbol{q}:\mathbb{R}^{4d-4-n}\times \mathbb{T}^n\rightarrow\mathcal{H}$
 is a finite-to-one covering map.
It follows that $\boldsymbol{q}_*(\pi_1(\mathbb{R}^{4d-4-n}\times\mathbb{T}^n))$ is a free abelian subgroup of rank $n$, 
having finite index in $\pi_1(\mathcal{H})$.
 If $\boldsymbol{q}_*(\pi_1(\mathbb{R}^{4d-4-n}\times\mathbb{T}^n))$ is normal and maximal abelian (yet to be characterized from the dynamical aspects), then $\pi_1(\mathcal{H})$ is a crystallographic group. This is the point where the global topology of hyperbolic components in the Cantor circle locus is related to the  crystallographic group.

\end{document}